\documentclass[a4paper,10pt]{article}
\usepackage[english]{babel}
\usepackage{graphicx}
\usepackage[margin=.3cm]{caption}
\usepackage[utf8]{inputenc}
\usepackage[T1]{fontenc}
\usepackage{lmodern}
 \usepackage[normalem]{ulem}
 \usepackage{bbm}
 \usepackage{enumerate}
 \usepackage{amsmath}
 \usepackage{dsfont}
 \usepackage{amssymb}
 \usepackage[mathscr]{eucal}
 \usepackage{mathtools}
 \usepackage{hyperref}

\usepackage{amsthm}
\theoremstyle{plain}

\newtheorem{thm}{Theorem}[section]
\newtheorem{lemma}{Lemma}[section]
\newtheorem{conj}[thm]{Conjecture}
\newtheorem{prop}[thm]{Proposition}
\newtheorem{coro}[thm]{Corollary}
\newtheorem{fact}{Fact}
\newtheorem{question}{Question}

\theoremstyle{remark}
\newtheorem*{rem}{Remark}

\theoremstyle{definition}

\newtheorem*{definition}{Definition of the finite-size criterion}

\newcommand{\N}{\ensuremath{\mathbb N}}
\newcommand{\Z}{\ensuremath{\mathbb Z}}
\newcommand{\bul}{\ensuremath{\bullet}}
\newcommand{\R}{\ensuremath{\mathbb R}}
\newcommand{\grp}{\ensuremath{G}}
\newcommand{\Grp}{\ensuremath{{\mathbb G}}}
\newcommand{\Graph}{\mathfrak{G}}
\newcommand{\gr}[1]{\mathsf{Graph}\left(#1\right)}

\newcommand{\gre}[2][e]{\mathsf{Graph}_{\mathbf{#1}}(#2)}

\newcommand{\lr}[1][]{\xleftrightarrow{#1\:}}
\newcommand{\nlr}[1][]{\overset{#1\:}{\longleftrightarrow} \kern -17pt
  \times \kern +7 pt}
\newcommand{\lrc}[2][]{\xleftrightarrow[#2]{#1}}
\renewcommand{\P}{\mathrm{{\mathbf P}}}
\newcommand{\Px}[1]{\mathrm{{\mathbf P}} \left [ #1 \right ]}

\newcommand{\Pp}[2][p]{\mathrm{{\mathbf P}}_{\!#1} \left \lbrack #2 \right \rbrack}
\renewcommand{\emph}{\textbf}

\mathtoolsset{showonlyrefs}

\date{\today} \author{S\'ebastien {\sc Martineau}\thanks{Both
    autors have been supported by the ANR grant MAC2
(ANR-10-BLAN-0123).},
  Vincent {\sc Tassion}\footnotemark[1]} \title{ \LARGE Locality of percolation for
  abelian Cayley graphs}
\begin{document}

\maketitle

\begin{abstract}
  We prove that the value of the critical probability for percolation
  on an abelian Cayley graph is determined by its local structure.
  This is a partial positive answer to a conjecture of Schramm: the
  function $\mathrm{p_c}$ defined on the set of Cayley graphs of
  abelian groups of rank at least $2$ is continuous for the
  Benjamini-Schramm topology. The proof involves group-theoretic tools
  and a new block argument.
\end{abstract}

\section{Introduction}

In the  paper \cite{bs1}, Benjamini and Schramm launched
the study of percolation in the general setting of transitive graphs.
Among the numerous questions that have been studied in this setting
stands the question of locality: roughly, ``does the value of the
critical probability depend only on the local structure of the
considered transitive graph ?'' This question emerged in \cite{bnp}
and is formalized in a conjecture attributed to Oded Schramm. In the same
paper, the particular case of (uniformly non-amenable) tree-like
graphs is treated.

In the present paper, we study the question of locality in the context
of abelian groups. 
\begin{itemize}
\item Instead of working in the geometric setting of transitive
  graphs, we employ the vocabulary of groups --- or more precisely of
  \emph{marked groups}, as presented in section~\ref{markedgroups}.
  This allows us to use
  additional tools of algebraic nature, such as quotient maps, that are
  crucial to our approach. These tools could be  useful to tackle
  Schramm's conjecture in a more general framework than the one
  presented in this paper, e.g.\@ Cayley graphs of nilpotent groups.

\item We extend renormalization techniques developed in \cite{gm} by
  Grimmett and Marstrand for the study of percolation on $\Z^d$
  (equipped with its standard graph structure). The Grimmett-Marstrand
  theorem answers positively the question of locality for the
  $d$-dimensional hypercubic lattice. With little extra effort, one
  can give a positive answer to Schramm's conjecture in the context
  of abelian groups, under a symmetry assumption. Our main achievement
  is to improve the understanding of supercritical bond percolation on
  general abelian Cayley graphs: such graphs do not have enough
  symmetry for Grimmett and Marstrand's arguments to apply directly.
  The techniques we develop here may be used to extend other results
  of statistical mechanics from symmetric lattices to lattices
    which are not stable under any reflection.
\end{itemize}

\subsection{Statement of Schramm's conjecture}
The following paragraph presents the vocabulary needed to state
Schramm's conjecture (for more details, see \cite{bnp}).

\paragraph{Transitive graphs}

We recall here some standard definitions from graph theory. A graph is
said to be \emph{transitive} if its automorphism group acts
transitively on its vertices. Let $\Graph$ denote the space of
(locally finite, non-empty, connected) transitive graphs considered up
to isomorphism. By abuse of notation, we will identify a graph with
its isomorphism class. Take $\mathcal{G}\in \Graph$ and $o$ any vertex
of $\mathcal{G}$. Then consider the \emph{ball} of radius $k$ (for the
graph distance) centered at~$o$, equipped with its graph structure and
rooted at $o$. Up to isomorphism of rooted graphs, it is independent
of the choice of $o$, and we denote it by $\mathcal
B_{\mathcal{G}}(k)$. If $\mathcal{G},\mathcal{H}\in\Graph$, we set the
distance between them to be $2^{-n}$, where
\begin{equation}
  n:=\max\{k:B_{\mathcal{G}}(k) \simeq B_{\mathcal{H}}(k)\}\in\N\cup\{\infty\}.
\end{equation}
This defines the \emph{Benjamini-Schramm
  distance} on the set $\Graph$. It was introduced in \cite{bs2} and
\cite{bnp}.

\paragraph{Locality in percolation theory}

We will use the standard definitions from percolation theory and refer
to \cite{grim} and \cite{lp} for background on the subject. To any
$\mathcal{G} \in \Graph$ corresponds a critical parameter
$\mathrm{p_c}(\mathcal{G})$ for i.i.d.\@ bond percolation. One can see
$\mathrm{p_c}$ as a function from $\Graph$ to $[0,1]$. The locality
question is concerned by the continuity of this function.
\begin{question}[Locality of percolation]\label{ques:locality}
  Consider a sequence of transitive graphs $(\mathcal
  G_n)$ that converges to a limit $\mathcal G$.
  \begin{center}
    Does the convergence $\displaystyle
    \mathrm{p_c}(\mathcal G_n) \xrightarrow[{n\to
      \infty}]{} \mathrm{p_c}(\mathcal G)$ hold?
  \end{center}
\end{question}
With this formulation, the answer is negative. Indeed, for the usual
graph structures, the following convergences hold:
\begin{itemize}
\item $\left(\Z/{n\Z}\right)^2\xrightarrow[{n\to \infty}]{} \Z^2$,

\item$\Z/{n\Z}\times \Z\xrightarrow[{n\to
    \infty}]{} \Z^2$.
\end{itemize} 
In both cases, the critical parameter is constant equal to $1$
all along the sequence and jumps to a non trivial value in the limit.
The following conjecture, attributed to Schramm and formulated
in~\cite{bnp}, states that Question~\ref{ques:locality} should have a
positive answer whenever the previous obstruction is avoided.

\begin{conj}[Schramm]\label{conj:Schramm}
  Let $\mathcal{G}_n\xrightarrow[{n\to\infty}]{}\mathcal{G}$ denote a
  converging sequence of transitive graphs. Assume that $\sup_n
  \mathrm{p_c}(\mathcal{G}_n)<1$. Then $
  \mathrm{p_c}(\mathcal{G}_n)\xrightarrow[{n\to
    \infty}]{}\mathrm{p_c}(\mathcal{G}). $
\end{conj}

It is unknown whether $\sup_n \mathrm{p_c}(\mathcal{G}_n)<1$ is
equivalent or not to $\mathrm{p_c}(\mathcal{G}_n) < 1$ for all $n$. In
other words, we do not know if $1$ is an isolated point in the set of
critical probabilities of transitive graphs. Besides, no geometric
characterization of the probabilistic condition
$\mathrm{p_c}(\mathcal{G})<1$ has been established so far, which
constitutes part of the difficulty of Schramm's conjecture.

\subsection{The Grimmett-Marstrand theorem}

The following theorem, proved in~\cite{gm}, is an instance of locality
result. It was an important step in the comprehension of the
supercritical phase of percolation.
\label{sec:slabs-conv-symm}
\begin{thm}[Grimmett-Marstrand]
  Let $d\geq 2$. For the usual graph structures, the following
  convergence holds:
  \begin{equation}
    \mathrm{p_c} \left(\Z^2\times\{-n,\ldots,n\}^{d-2}\right)
    \underset{n\to\infty}{\longrightarrow} \mathrm{p_c}\left(\Z^{d}\right).
\end{equation}
\end{thm}
\begin{rem}
  Grimmett and Marstrand's proof covers more generally the case of
  edge structures on $\Z^d$ that are  invariant under both translation
  and reflection.
\end{rem}
The graph $\Z^2\times\{-n,\ldots,n\}^{d-2}$ is not transitive, so the
result does not fit exactly into the framework of the previous
subsection. However, as remarked in~\cite{bnp}, one can easily deduce
from it the following statement:
\begin{equation}
  \label{eq:1}
  \mathrm{p_c} \left(\Z^2\times\left(\frac{\Z}{n\Z}\right)^{d-2}\right)
  \underset{n\to\infty}{\longrightarrow} \mathrm{p_c}\left(\Z^{d}\right).
\end{equation}

Actually, after having introduced the space of marked abelian groups,
we will see in section~\ref{specialcase} that one can deduce from the
Grimmett-Marstrand theorem a statement that is much stronger than
convergence~\eqref{eq:1}. We will be able to prove that
$\mathrm{p_c}(\Z^d)=\lim \mathrm{p_c}(\mathcal G_n)$ for any sequence
of abelian Cayley graphs $\mathcal G_n$ converging to $\Z^d$ with
respect to the Benjamini-Schramm distance.

\subsection{Main result}
\label{sec:main-results}

In this paper we prove the following theorem, which provides a
positive answer to Question \ref{ques:locality} in the particular case
of Cayley graphs of abelian groups (see definitions in
section~\ref{markedgroups}).

\begin{thm}\label{continuitythmtwo}
  Consider a sequence $(\mathcal G_n)$ of Cayley graphs of abelian
  groups satisfying $\mathrm{p_c}(\mathcal G_n)<1$ for all $n$. If the
  sequence converges to the Cayley Graph $\cal G$ of an abelian group,
  then
  \begin{equation}
    \label{eq:2}
    \displaystyle \mathrm{p_c}(\mathcal G_n) \xrightarrow[n\to
    \infty]{} \mathrm{p_c}(\mathcal G).
\end{equation}
\end{thm}

We now give three examples of application of this theorem. Let $d\ge 2$, fix a generating set $S$ of
$\Z^d$, and denote by $\mathcal G$ the associated Cayley graph of
$\Z^d$.

\begin{description}
\item[Example 1:] There exists a natural Cayley graph $\mathcal G_n$
  of $\Z^2\times\left(\frac{\Z}{n\Z}\right)^{d-2}$ that is covered by
  $\mathcal G$. For such sequence, the convergence~\eqref{eq:2} holds,
  and generalizes~\eqref{eq:1}.

\item[Example 2:] Consider the generating set of $\Z^d$ obtained by
  adding to  $S$ all the $n\cdot s$, for $s\in S$. The corresponding
  Cayley graph $\mathcal H_n$ converges to the Cartesian product $\cal G
  \times \cal G$, and we get
  \begin{equation}
         \mathrm{p_c}(\mathcal H_n) \xrightarrow[n \to  \infty]{}
    \mathrm{p_c}(\mathcal G \times  \mathcal G).
  \end{equation}
  
\item[Example 3:] Consider a sequence of vectors $x_n\in \Z^d$ such
  that $\lim|x_n|=\infty$, and write $\mathcal G_n$ the Cayley graph
  of $\Z^d$ constructed from the generating set $S\cup \{x_n\}$. Then
  the following convergence holds:
  \begin{equation}
    \mathrm{p_c}(\mathcal G_n) \xrightarrow[n \to  \infty]{}
    \mathrm{p_c}(\mathcal G \times  \Z).
  \end{equation}
\end{description}

The content of Example 2 was obtained  in \cite{lss} when $\cal G$  is
the canonical Cayley graph of $\Z^d$, based on Grimmett-Marstrand theorem. In the statement above,
$\mathcal G$ can be any Cayley graph of $\Z^d$, and Grimmett-Marstrand
theorem cannot be applied without additional symmetry assumption.

\subsection{Questions}
\label{sec:questions}

In this paper, we work with abelian groups because their structure is
very well understood. An additional important feature is that the net
formed by large balls of an abelian Cayley graph has roughly the same
geometric structure as the initial graph. Since nilpotent groups also
present these characteristics, the following question appears as a
natural step between Theorem~\ref{continuitythmtwo} and
Question~\ref{ques:locality}.
\begin{question}\label{ques:nilpotent}
  Is it possible to extend Theorem~\ref{continuitythmtwo} to nilpotent
  groups?
\end{question}

This question can also be asked for other models of statistical
mechanics than Bernoulli percolation. In questions~\ref{ques:3} and
\ref{ques:4}, we mention two other natural contexts where the locality
question can be asked.

Theorem~2.1 of \cite{b} states that locality holds for the
critical temperature of the Ising model for the hypercubic lattice.
This suggests the following question.

\begin{question}
  \label{ques:3}
  Is it possible to prove Theorem~\ref{continuitythmtwo} for the
  critical temperature of the Ising model instead of $\mathrm{p_c}$ ?
\end{question}

Define $c_n$ as the number of self-avoiding walks starting from a
fixed root of a transitive graph $\mathcal G$. By
sub-multiplicativity, the sequence $c_n^{1/n}$ converges to a limit
called the connective constant of $\mathcal G$. In this context, the
following question was raised by I. Benjamini \cite{ib}:  
\begin{question}
\label{ques:4}
Does the connective constant depend continuously on the
considered infinite transitive graph?
\end{question}

\subsection{Organization of the paper}
\label{sec:organization-paper}

Section~\ref{markedgroups} presents the material on marked abelian
groups that will be needed to establish
Theorem~\ref{continuitythmtwo}. In section~\ref{sec:heuristics}, we
explain the strategy of the proof, which splits into two main lemmas.
Sections~\ref{supercritical} and \ref{sec:renormalization} are each
devoted to the proof of one of these lemmas.
 
We drive the attention of the interested reader to
Lemma~\ref{lem:path}. Together with the uniqueness of the infinite
cluster, it allows to avoid the construction of ``seeds'' in Grimmett
and Marstrand's approach.

\section{Marked abelian groups and locality}
\label{markedgroups}

In this section, we present the space of marked abelian groups and
show how problems of Benjamini-Schramm continuity for abelian Cayley
graphs can be reduced to continuity problems for marked abelian group.
Then, we provide a first example illustrating the use of marked
abelian groups in proofs of Benjamini-Schramm continuity. Finally,
section~\ref{sec:heuristics} presents the proof of
Theorem~\ref{continuitythmone}, which is the marked group version of
our main theorem.

General marked groups are introduced in \cite{grig}. Here, we only
define marked groups and Cayley graphs in the \emph{abelian} setting,
since we do not need a higher level of generality.

\subsection{The space of marked abelian groups}
\label{sec:space-marked-abelian}
\renewcommand{\Grp}{\mathbf{G}} Let $d$ denote a positive integer. A
\emph{($d$-)marked abelian group} is the data of an abelian group
together with a generating $d$-tuple $(s_1,\ldots,s_d)$, up to
isomorphism\footnote{$(G;s_1,\dots,s_d)$ and $(G';s'_1,\dots,s'_d)$
  are isomorphic if there exists a group isomorphism from $G$ to $G'$
  mapping $s_i$ to $s'_i$ for all $i$.}. We write $\Grp_d$ the set of
the $d$-marked abelian groups. Elements of $\Grp_d$ will be denoted by
$[\grp ; s_1,\ldots,s_d ]$ or $\grp^\bullet$, depending on whether we
want to insist on the generating system or not. Finally, we write
$\Grp$ the set of all the marked abelian groups: it is the disjoint
union of all the $\Grp_d$'s.

\paragraph{Quotient of a marked abelian group}
Given a marked abelian group $\grp^\bul=[G;s_1,\ldots,s_d]$ and a
subgroup $\Lambda$ of $G$, we define the \emph{quotient}
$\grp^\bul/\Lambda$ by
\begin{equation}
\grp^\bul/\Lambda=[\grp/\Lambda;\overline{s_1},\dots,\overline{s_d}],
\end{equation}
where $(\overline{s_1},\dots,\overline{s_d})$ is the image of
$(s_1,\ldots,s_d)$ by the canonical surjection from $G$ onto
$G/\Lambda$. Quotients of marked abelian groups will be crucial to
define and understand the topology of the set of marked abelian
groups. In particular, for the topology defined below, the quotients
of a marked abelian group $\grp^\bul$ forms a neighbourhood of it.

\paragraph{The topology} We first define the topology on $\Grp_d$. Let
$\mathbf{\delta}$ denote the canonical generating system of $\Z^d$. To
each subgroup $\Gamma$ of $\Z^d$, we can associate an element of
$\Grp_d$ via the mapping
\begin{equation}
    \Gamma \longmapsto[\Z^d; \mathbf{\delta}]/\Gamma.\label{eq:3}
\end{equation}
One can verify that the mapping defined by~\eqref{eq:3} realizes a
bijection from the set of the subgroups of $\Z^d$ onto $\Grp_d$. This
way, $\Grp_d$ can be seen as a subset of~$\{0,1\}^{\Z^d}$. We consider
on $\Grp_d$ the topology induced by the product topology on
$\{0,1\}^{\Z^d}$. This makes of $\Grp_d$ a Hausdorff compact space.
Finally, we equip $\Grp$ with the topology generated by the open
subsets of the $\Grp_d$'s. (In particular, $\Grp_d$ is an open subset
of $\Grp$.)

Let us illustrate the topology with three examples of converging sequences: 
\begin{itemize}
\item $[\Z/n\Z;1]$ converges to $[\Z;1]$.
\item $[\Z;1,n,\dots,n^d]$ converges to $[\Z^d;\mathbf{\delta}]$.
\item $[\Z;1,n,n+1]$ converges to 
$[\Z^2;\mathbf{\delta}_1,\mathbf{\delta}_2,\mathbf{\delta}_1+\mathbf{\delta}_2]$.
\end{itemize}

\paragraph{Cayley graphs}
\label{sec:cayley-graphs}
Let $\grp^\bul=[\grp;s_1,\dots,s_d]$ be a marked abelian group. Its
Cayley graph, denoted $\mathsf{Cay}(\grp^\bul)$, is defined by taking
$\grp$ as vertex-set and declaring $a$ and $b$ to be neighbours if
there exists $i$ such that $a=b \pm s_i$. It is is uniquely defined up
to graph isomorphism. We write $B_{\grp^\bul}(k)\subset G$ the ball of
radius $k$ in $\mathsf{Cay}(\grp^\bul)$, centered at $0$.

\paragraph{Converging sequences of marked abelian groups}
\label{sec:conv-sequ-mark}
In the rest of the paper, we will use the topology of $\Grp$ through
the following proposition, which gives a geometric flavour to the
topology. In particular, it will allow to do the connection with the
Benjamini-Schramm topology through corollary~\ref{continuity}.
\begin{prop}
  \label{prop:conv-seq}
  Let $(\grp^\bul_n)$ be a sequence of marked abelian groups that
  converges to
  some $\grp^\bul$. Then, for any integer $k$, the following
  holds for $n$ large enough:
  \begin{enumerate}
  \item $\grp^\bul_n$ is of the form $\grp^\bul/\Lambda_n$, for some
    subgroup $\Lambda_n$ of $G$, and
  \item $\Lambda_n\cap B_{G^\bul}(k)=\{0\}$.
  \end{enumerate}
\end{prop}
\begin{proof}
  Let $d$ be such that $G^\bul \in \Grp_d$. For $n$ large enough, we
  also have $G_n^\bul\in\Grp_d$. Let $\Gamma$ (resp. $\Gamma_n$)
  denote the unique subgroup of $\Z^d$ that corresponds to $G^\bul$
  (resp. $\grp^\bul_n$) via bijection \eqref{eq:3}. The group $\Gamma$
  is finitely generated: we consider $F$ a finite generating subset of
  it. Taking $n$ large enough, we can assume that $\Gamma_n$ contains
  $F$, which implies that $\Gamma$ is a subgroup $\Gamma_n$. We have
  the following situation
  \begin{equation}
    \Z^d \xrightarrow[]{\varphi} \Z^d/\Gamma \xrightarrow[]{\psi_n}
    \Z^d/\Gamma_n.
  \end{equation}
  Identifying $\grp$ with $\Z^d/\Gamma$ and taking
  $\Lambda_n=\mathrm{ker}\, \psi_n=\Gamma_n/\Gamma$, we obtain the
  first point of the proposition.

  By definition of the topology, taking $n$ large enough ensures that
  $\Gamma_n\cap B_{\Z^d}(k)=\Gamma\cap B_{\Z^d}(k)$. We have 
  \begin{align}
    B_{\Z^d/\Gamma}(k)\cap \Lambda_n&=\varphi( B_{\Z^d}(k)\cap\Gamma_n)\\
    &=\varphi(B_{\Z^d}(k)\cap\Gamma)\\
    &=\{0\}.
  \end{align}
  This ends the proof of the second point.
\end{proof}

\begin{coro}
  \label{continuity} The mapping $\mathsf{Cay}$ from $\Grp$ to
  $\Graph$ that associates to a marked abelian group its Cayley graph
  is continuous.
\end{coro}

\subsection{Percolation on marked abelian groups}
\label{sec:perc-mark-abel}
Via its Cayley graph, we can associate to each marked abelian group
$\grp^\bul$ a critical parameter
$\mathrm{p_c}^{\!\!\!\bul}(\grp^\bul):=\mathrm{p_c}(\mathsf{Cay}(\grp^\bul))$
for bond percolation. If $\grp^\bul$ is a marked abelian group, then
$\mathrm{p_c}^{\!\!\!\bul}(\grp^\bul)<1$ if and only if the rank of
$\grp$ is at least~$2$. (We commit the abuse of language of calling
\emph{rank} of an abelian group the rank of its torsion-free part.)
This motivates the following definition:
\begin{equation}
  \tilde\Grp=\left\{G^\bul\in \Grp \::
    \: \mathsf{rank}(G)\ge 2 \right\}.
\end{equation}
In the context of marked abelian groups, we will prove the following
theorem:
\begin{thm}\label{continuitythmone}
  Consider $\grp_n^{\bul} \longrightarrow \grp^{\bul}$ a converging
  sequence in $\tilde\Grp$. Then,
  \begin{equation}
    \displaystyle \mathrm{p_c}^{\!\!\!\bul}(\grp_n^{\bul}) \xrightarrow[n\to
    \infty]{} \mathrm{p_c}^{\!\!\!\bul}(\grp^\bul).
\end{equation} 
\end{thm}
 
Theorem~\ref{continuitythmone} above states that
$\mathrm{p_c}^{\!\!\!\bul}$ is continuous on $\tilde\Grp$. It seems a
priori weaker than Theorem~\ref{continuitythmtwo}. Nevertheless, the
following lemma allows us to deduce Theorem~\ref{continuitythmtwo}
from Theorem~\ref{continuitythmone}.
\begin{lemma}
  \label{lem:reduce} Let $G^\bul$ be an element of $\tilde\Grp$.
Assume it is a continuity point of the restricted function
\begin{equation}
  \mathrm{p_c}^{\!\!\!\bul}:\:\tilde{\Grp}\longrightarrow (0,1).
\end{equation} Then its associated Cayley graph
$\mathsf{Cay}(G^\bul)$ is a continuity point of the restricted
function
\begin{equation}
    \mathrm{p_c}:\mathsf{Cay}(\tilde\Grp)\longrightarrow (0,1).
  \end{equation}
\end{lemma}
\begin{proof} Assume, by contradiction, that there exists a sequence
  of marked abelian groups $\grp_n^\bul$ in $\tilde\Grp$ such that
  $\mathsf{Cay}(\grp_n^\bul)$ converges to some
  $\mathsf{Cay}(\grp^\bul)$ and
  $\mathrm{p_c}^{\!\!\!\bul}(\grp_n^\bul)$ stays away from
  $\mathrm{p_c}^{\!\!\!\bul}(\grp^\bul)$. Define $d$ to be the degree
  of $\mathsf{Cay}(\grp^\bul)$. Considering $n$ large enough, we can
  assume that all the $G_n^\bul$'s lie in the compact set
  $\bigcup_{d'\le d}\Grp_{d'}$. Up to extraction, one can then assume
  that $\grp_n^\bul$ converges to some marked abelian group
  $\grp_\infty^\bul$. This group must have rank at least 2. Since
  $\mathsf{Cay}$ is continuous,
  $\mathsf{Cay}(\grp^\bul)=\mathsf{Cay}(\grp_\infty^\bul)$ and
  Theorem~\ref{continuitythmone} is contradicted by the sequence
  $(\grp_n^\bul)$ that converges to $\grp_\infty^\bul$.
\end{proof} We will also use the following theorem, which is a
particular case of theorem 3.1 in \cite{bs1}.
\begin{thm}
  \label{quotient} Let $G^\bul$ be a marked abelian group and
  $\Lambda$ a subgroup of $G$. Then
$$\mathrm{p_c}^{\!\!\!\bul}(G^\bul/\Lambda) \geq \mathrm{p_c}^{\!\!\!\bul}(G^\bul).$$
\end{thm}

\subsection{A first continuity result}
\label{specialcase} In this section, we will prove
Proposition~\ref{thm:first-cont-result}, which is a particular case of
Theorem~\ref{continuitythmtwo}. We deem interesting to provide a short
independent proof of it. This proposition epitomizes the scope of
Grimmett-Marstrand results in our context. It also illustrates how
marked groups can appear as useful tools to deal with locality
questions. More precisely, Lemma~\ref{lem:reduce} reduces some
questions of continuity in the Benj\-amini-Schramm space to equivalent
questions in the space of marked abelian groups, where the topology
allows to employ methods of algebraic nature.

\begin{prop}
\label{thm:first-cont-result}
Let $(\grp_n^\bullet)$ be a sequence in $\tilde\Grp$. Assume that
$\grp_n^\bullet\xrightarrow[n\to\infty]{}[\Z^d;\mathbf{\delta}]$,
where $\mathbf{\delta}$ stands for the canonical generating system of
$\Z^d$. Then
\begin{equation}
  \mathrm{p_c}^{\!\!\!\bul}(\grp_n^\bullet) \xrightarrow[n\to
  \infty]{}\mathrm{p_c}^{\!\!\!\bul}([\Z^d;\mathbf{\delta}]).
\end{equation}
\end{prop}

\begin{proof}
  Since $\Grp_d$ is open, we can assume that $\grp_n^\bul$ belongs to
  it. It is thus a quotient of $[\Z^d;\mathbf{\delta}]$, and
  Theorem~\ref{quotient} gives
$$
\liminf \mathrm{p_c}^{\!\!\!\bul}(\grp_n^\bul)\geq
\mathrm{p_c}^{\!\!\!\bul}([\Z^d;\mathbf{\delta}]).
$$
To establish the other semi-continuity, we will show that the Cayley
graph of $\grp_n^\bul$ eventually contains $\Z^2\times\{0,\ldots,K\}$
as a subgraph (for $K$ arbitrarily large), and conclude by applying
Grimmett-Marstrand theorem.
  
Let us denote $\Gamma_n$ the subgroup of $\Z^d$ associated to
$\grp_n^\bullet$ via bijection~\eqref{eq:3}. We call coordinate plane
a subgroup of $\Z^d$ generated by two different elements of the
canonical generating system of $\Z^d$.

\begin{lemma}
  \label{lem:slab} For any integer $K$, for $n$ large enough, there
  exists a coordinate plane $\Pi$ satisfying
  \begin{equation}
    \left(\Pi+ B_{\Z^d}(0,2K+1)\right)\cap\Gamma_n =\{0\}.
\end{equation}
\end{lemma}

\begin{small}
\begin{proof}[Proof of Lemma~\ref{lem:slab}]
  To establish Lemma~\ref{lem:slab}, we proceed by contradiction. Up
  to extraction, we can assume that there exists some $K$ such that
\begin{equation}
  \text{for all $\Pi$,} \quad  (\Pi +  B_{\Z^d}(0,2K+1))\cap\Gamma_n \not=\{0\}.\label{eq:4}
\end{equation}
We denote by $v_n^\Pi$ a non-zero element of $(\Pi +
B_{\Z^d}(0,2K+1))\cap\Gamma_n $. Up to extraction, we can assume that,
for all $\Pi$, the sequence $v^\Pi_n/{\|v^\Pi_{n}\|}$ converges to
some $v_{\Pi}$. (The vector space $\R^d$ is endowed with an arbitrary
norm $\|~\|$.) Since $\Gamma_n$ converges pointwise to $\{0\}$, for
any $\Pi$, the sequence $\|v_n^\Pi\|$ tends to infinity. This entails,
together with equation~\eqref{eq:4}, that $v_\Pi$ is contained in the
real plane spanned by $\Pi$. The incomplete basis theorem implies that
the vector space spanned by the $v_\Pi$'s has dimension at least
$d-1$. By continuity of the minors, for $n$ large enough, the vector
space spanned by $\Gamma_n$ as dimension at least $d-1$. This entails
that, for $n$ large enough, $\Gamma_n$ has rank at least $d-1$, which
contradicts the hypothesis that $\Z^d/\Gamma_n$ has rank at least 2.
\end{proof}
\end{small}

For any $K$, provided that $n$ is large enough, one can see
$\Z^2\times\{-K,\ldots,K\}^{d-2}$ as a subgraph of
$\mathsf{Cay}(\grp_n^\bul)$. (Restrict the quotient map from $\Z^d$ to
$\grp_n^\bul$ to the $\Pi + B_{\Z^d}(0,K)$ given by
Lemma~\ref{lem:slab} and notice that it becomes injective.) It results
from this that $$\limsup \mathrm{p_c}^{\!\!\!\bul}(\grp_n^\bul)\leq
\mathrm{p_c}(\Z^2\times\{-K,\ldots,K\}^{d-2}).$$The right-hand side
goes to $ \mathrm{p_c}^{\!\!\!\bul}([\Z^d;\mathbf{\delta}])$ as $K$
goes to infinity, by Grimmett-Marstrand theorem. This establishes the
second semi-continuity.
\end{proof}

\begin{rem}
  Proposition~\ref{thm:first-cont-result} states exactly what
  Grimmett-Marstrand theorem implies in our setting. Together with
  Lemma~\ref{lem:reduce}, it entails that the hypercubic lattice is a
  continuity point of $\mathrm{p_c}$ on $\mathsf{Cay}(\tilde\Grp)$.
  Without additional idea, one could go a bit further: the proof of
  Grimmett and Marstrand adjusts directly to the case of Cayley graphs
  of $\Z^d$ that are stable under reflections relative to coordinate
  hyperplanes. This statement also has a counterpart analog to
  Proposition~\ref{thm:first-cont-result}. Though, we are still far
  from Theorem~\ref{continuitythmone}, since Grimmett-Marstrand
  theorem relies heavily on the stability under reflection. In the
  rest of the paper, we solve the locality problem for general abelian
  Cayley graphs. We do so directly in the marked abelian group
  setting, and do not use a ``slab result'' analog to
  Grimmett-Marstrand theorem.
\end{rem}

\subsection{Proof of Theorem~\ref{continuitythmone}}

\label{sec:heuristics}

The purpose of this section is to reduce the proof of
Theorem~\ref{continuitythmone} to the proof of two lemmas
(Lemma~\ref{lem:finite-criterion} and
Lemma~\ref{lem:renormalization}). These are respectively established
in sections \ref{supercritical} and \ref{sec:renormalization}.

As in section~\ref{specialcase}, it is the upper semi-continuity of
$\mathrm{p_c}^{\!\!\!\bul}$ that is hard to establish: given $G^\bul$
and $p>\mathrm{p_c}^{\!\!\!\bul}(G^\bul)$, we need to show that the
parameter $p$ remains supercritical for any element of $\tilde\Grp$
that is close enough to $G^\bul$. To do so, we will characterize
supercriticality by using \emph{a finite-size criterion}, that is a
property of the type ``$\Pp{\mathcal{E}_N} > 1-\eta$'' for some event
$\mathcal{E}_N$ that depends only on the states of the edges in the
ball of radius $N$. The finite-size criterion we use is denoted by
$\mathscr{FC}(p,N,\eta)$ and characterizes supercriticality through
lemmas~\ref{lem:finite-criterion} and \ref{lem:renormalization}. Its
definition involving heavy notation, we postpone it to
section~\ref{sec:stab-lemma-under}.

First, we work with a fixed marked abelian group $G^\bul$. Assuming
that $p>\mathrm{p_c}^{\!\!\!\bul}(G^\bul)$, we construct in its Cayley
graph a box that exhibits nice connection properties with high
probability. This is formalized by Lemma~\ref{lem:finite-criterion}
below, which will be proved in section~\ref{supercritical}.

\begin{lemma}
  \label{lem:finite-criterion}
  Let $G^\bul\in\tilde\Grp$. Let $p>\mathrm{p_c}^{\!\!\!\bul}(G^\bul)$
  and $\eta>0$. Then, there exists $N$ such that $G^\bul$ satisfies
  the finite-size criterion $\mathscr{FC}(p,N,\eta)$.
\end{lemma}
Then, take $H^{\bul}=G^\bul/\Lambda$ a marked abelian group that is
close to $G^\bul$. Since $\mathsf{Cay}(G^\bul)$ and
$\mathsf{Cay}(H^\bul)$ have the same balls of large radius, the finite
criterion is also satisfied by $H^\bul$. This enables us to prove that
there is also percolation in $\mathsf{Cay}(H^\bul)$. As in Grimmett
and Marstrand's approach, we will not be able to prove that
percolation occurs in $\mathsf{Cay}(H^\bul)$ for the same parameter
$p$, but we will have to slightly increase the parameter. Here comes a
precise statement, established in section~\ref{sec:renormalization}.

\begin{lemma}
  \label{lem:renormalization}
  Let $G^\bul \in \tilde\Grp$. Let
  $p>\mathrm{p_c}^{\!\!\!\bul}(G^\bul)$ and $\delta>0$. Then there
  exists $\eta>0$ such that the following holds: if there exists $N$ such
  that $G^\bul$ satisfies the finite-size criterion $\mathscr{FC}(p,N,\eta)$,
  then $\mathrm{p_c}(H^\bul)<p+\delta$ for any marked abelian
  group  $H^\bul$  close enough to $G^\bul$.  
  \end{lemma}
Assuming these two lemmas, let us prove Theorem~\ref{continuitythmone}.

\begin{proof}[Proof of Theorem~\ref{continuitythmone}]

  Let $\grp_n^{\bul} \xrightarrow[n\to \infty]{} \grp^\bul$ denote a
  converging sequence of elements of $\tilde\Grp$. Our goal is to
  establish that $\mathrm{p_c}^{\!\!\!\bul}(\grp_n^\bul)
  \xrightarrow[n\to \infty]{} \mathrm{p_c}^{\!\!\!\bul}(\grp^\bul)$.

  For $n$ large enough, $\grp_n^\bul$ is a quotient of $\grp^\bul$.
  (See Proposition~\ref{prop:conv-seq}.) By Theorem~\ref{quotient},
  for $n$ large enough, $\mathrm{p_c}^{\!\!\!\bul}(\grp^\bul)\leq
  \mathrm{p_c}^{\!\!\!\bul}(\grp^\bul_n)$. Hence, we only need to
  prove that $\limsup \mathrm{p_c}^{\!\!\!\bul}(\grp_n^\bul) \leq
  \mathrm{p_c}^{\!\!\!\bul}(\grp^\bul).$

  Take $p> \mathrm{p_c}$ and $\delta > 0$. By
  Lemma~\ref{lem:finite-criterion}, we can pick $N$ such that
  $\mathscr{FC}(p,N,\eta)$ is satisfied. Lemma~\ref{lem:renormalization}
  then guarantees that, for $n$ large enough,
  $\mathrm{p_c}^{\!\!\!\bul}(\grp^\bul_n)\leq p+\delta$, which ends
  the proof.
\end{proof}

\section{Proof of Lemma~\ref{lem:finite-criterion}}

\label{supercritical}
Through the entire section, we fix:
\begin{itemize}
\item[-] $\grp^\bul \in \tilde\Grp$ a marked abelian group of rank
greater than two,
\item[-] $p\in (\mathrm{p_c}^{\!\!\!\bul}(G^\bul),1)$,
\item[-] $\eta>0$.
\end{itemize}
We write $\grp^\bul$ under the form $\left \lbrack \Z^r \times T; \: S
\right \rbrack$, where $T$ is a finite abelian group. Let
$\mathcal{G}=(V,E)=(\Z^r\times T, E)$ denote the Cayley graph
associated to $\grp^\bul$. Paths and percolation will always be
considered relative to this graph structure.

\subsection{Setting and notation}

\subsubsection{Between continuous and discrete}

An element of $\Z^r\times T$ will be
written
\begin{equation}
x=({x}_\text{free},{x}_\text{tor}).
\end{equation}
For the geometric reasonings, we will use linear algebra tools. (The
vertex set --- $\Z^r\times T$ --- is roughly $\R^r$.) Endow $\R^r$ with
its canonical Euclidean structure. We denote by $\|~\|$ the associated
norm and $\mathbb{B}(v,R)$ the closed ball of radius $R$ centered at
$v\in\R^r$. If the center is $0$, this ball may be denoted by
$\mathbb{B}(R)$. Set $R_S := \max_{s\in S} \|s_\text{free}\|$. In
$\mathcal G$, we define for $k>0$ 
\begin{align}
 B(k)&:=\{x:\|x_{\mathrm{free}}\|\leq k R_S\}\\
 &\:=(\mathbb{B}(kR_S)\cap \Z^d )\times T.
\end{align}

Up to section~\ref{sec:stab-lemma-under}, we fix an orthornomal basis
$\mathbf{e}=(e_1,\dots,e_d)$ of $\R^r$. Define
\begin{equation}
\begin{array}{lccl}
\pi_{\mathbf e} : & \R^r & \longrightarrow &  \R^2 \\
    & \sum_{i=1}^r x_i e_i & \longmapsto & (x_1,x_2). \end{array}
\end{equation}
 We now define the
function $\mathsf{Graph}$, which allows us to move between the
continuous space $\R^2$ and the discrete set $V$. It associates to
each subset $X$ of $\R^2$ the subset of $V$ defined by
\begin{equation}
\label{eq:5}
\gr{X}:= \left ( \left( \pi_{\mathbf e}^{-1}(X)+ \mathbb{B}(R_S)\right) \cap \Z^r
\right ) \times T.
\end{equation}
In section~\ref{sec:stab-lemma-under}, we will have to consider
different bases. To insist on the dependence on $\mathbf{e}$, we will write
$\mathsf{Graph}_\mathbf{e}$.

If $a$ and $b$ belong to $\R^2$, we will consider the segment $[a,b]$
and the parallelogram $[a,b,-a,-b]$ spanned by $a$ and $b$ in $\R^2$,
defined respectively by
\begin{align}
  &[a,b]=\{\lambda a + (1-\lambda)b \;;\; 0\leq\lambda\leq 1\} \text{ and}\\
  &[a,b,-a,-b]=\{\lambda a + \mu b \; ;\; |\lambda|+|\mu|\leq 1 \}
\end{align}
Write then $L(a,b):=\gr{[a,b]}$ and $R(a,b):=\gr{[3a,3b,-3a,-3b]}$ the
corresponding subsets of $V$.

The following lemma illustrates one important property of the function $\mathsf{Graph}$
connecting continuous and discrete.

\begin{lemma}
\label{contdiscrete}
  Let $X\subset \R^2$. Let $\gamma$ be a finite path of length $k$
  in $\mathcal{G}$. Assume that $\gamma_0\in \gr{X}$ and $\gamma_k
  \not\in \gr{X}$. Then the support of $\gamma$ intersects
  $\gr{\partial X}$.
\end{lemma}

\begin{proof} It suffices to show that if $x$ and $y$ are two
  neighbours in $\mathcal{G}$ such that $x\in \gr{X}$ and $y \notin
  \gr{X}$, then $x$ belongs to $\gr{\partial X}$. By definition of
  $\mathsf{Graph}$, we have
  $x_{\mathrm{free}}\in\pi^{-1}(X)+\mathbb{B}(R_S )$, which can be
  restated as
  \begin{equation}
    \label{eq:6}
    \pi\left(\mathbb{B}(x_{\mathrm{free}},R_S)\right)\cap X
    \neq \emptyset.
  \end{equation}
  By definition of $R_S$, we have $y_{\mathrm{free}}\in
  \mathbb{B}(x_{\mathrm{free}},R_S) $ and our assumption on $y$
  implies that $\pi(y_{\mathrm{free}})\notin X$, which gives
  \begin{equation}
    \label{eq:7}
    \pi\left(\mathbb{B}(x_{\mathrm{free}},R_S)\right)\cap ~^cX
    \neq \emptyset.
  \end{equation}
  Since $\pi\left(\mathbb{B}(x_{\mathrm{free}},R_S)\right)$ is
  connected, \eqref{eq:6} and \eqref{eq:7} implies that
  \begin{equation}
    \pi\left(\mathbb{B}(x_{\mathrm{free}},R_S)\right)\cap \partial X
    \neq \emptyset
  \end{equation}
  which proves that $x$ belongs to $\gr{\partial X}$.
 \end{proof}

 \subsubsection{Percolation toolbox}

 \paragraph{Probabilistic notation} We denote by $\P_p$ the law of independent bond
 percolation of parameter $p\in[0,1]$ on $\mathcal{G}$.

 \paragraph{Connections} Let $A$, $B$ and $C$ denote three subsets of
 $V$. The event ``there exists an open path intersecting $A$ and $B$
 that lies in $C$'' will be denoted by ``$A\lr[C] B$''. The event
 ``restricting the configuration to $C$, there exists a unique
 component that intersects $A$ and $B$'' will be written
 ``$A\lr[!C!]B$''. The event ``there exists an infinite open path that
 touches $A$ and lies in $C$ will be denoted by ``$A \lr[C] \infty$''.
 If the superscript $C$ is omitted, it means that $C$ is taken to be
 the whole vertex set.

 This paragraph contains the percolation results that will be needed
 to prove Theorem~\ref{continuitythmone}. The following lemma,
 sometimes called ``square root trick'', is a straightforward
 consequence of Harris-FKG inequality.

\begin{lemma}
  \label{lemFKG}
  Let $\mathcal A$ and $\mathcal B$ be two increasing events. Assume
  that $\displaystyle \Pp{\cal A} \ge \Pp{\cal B}$. Then, the
  following inequality holds:
  \begin{equation}
    \Pp{\cal A} \geq 1-\left(1-\Pp{\mathcal A \cup \mathcal B }\right)^{1/2}.
  \end{equation}
\end{lemma}
The lemma above is often used when $\Pp{\mathcal A} =\Pp{\mathcal B}$,
in a context where the equality of the two probabilities is provided
by symmetries of the underlying graph (see \cite{grim}). This slightly
generalized version allows to link geometric properties to
probabilistic estimates whithout any symmetry assumption, as
illustrated by the following lemma.

\begin{lemma}\label{lem:connections}
  Let $a$ and $b$ be two points in $\R^2$. Let $A \subset V$ be a
  subset of vertices of $\mathcal G$. Assume that
  \begin{equation}
    \label{eq:8}
    \Pp{A\lr L(a,b)}>1-\varepsilon^2\text{ for some $\varepsilon>0$.}
  \end{equation}
  Then, there exists $u\in [a,b]$ such that  both $\Pp{ A
  \lr L(a,u)}$ and $\Pp{ A
  \lr L(u,b)}$ exceed $1-\varepsilon$.
\end{lemma}
\begin{rem}
  The same statement holds when we restrict the open paths to lie in a
  subset $C$ of $V$.
\end{rem}

\begin{proof}
  We can approximate the event estimated in inequality~\eqref{eq:8}
  and pick $k$ large enough such that
  \begin{equation}
     \Pp{A\lr L(a,b)\cap B(k) }>1-\varepsilon^2.
  \end{equation}
 
  The set $L(a,b)\cap B(k)$ being finite, there are only finitely many
  different sets of the form $L(a,u)\cap B(k)$ for $u\in[a,b]$. We can
  thus construct $u_1,u_2\ldots,u_n \in [a,b]$ such that $u_1=a$ and
  $u_n=b$, and for all $1\le i <n$,
\begin{enumerate}
\item $[a,u_i]$ is a strict subset of $[a,u_{i+1}]$,
\item $L(a,b)\cap B(k)$ is the union of $L(a,u_i)\cap B(k)$ and
  $L(u_{i+1},b)\cap B(k)$.
\end{enumerate}
Assume that for some $i$, the following inequality holds:
 \begin{equation}
  \label{eq:9}
  \Pp { A\lr L(a,u_{i}) \cap B(k) } \ge \Pp { A\lr L(u_{i+1},b) \cap B(k)}.
\end{equation}
Lemma~\ref{lemFKG} then implies that 
\begin{equation}
\Pp {A\lr
    L(a,u_{i}) \cap B(k)} > 1- \varepsilon.
\end{equation}

If inequality~\eqref{eq:9} never holds (resp. if it holds for all
possible $i$), then $A$ is connected to $L(\{a\})$ (resp. to
$L(\{b\})$) with probability exceeding $1-\varepsilon$. In these two cases, the
conclusion of the lemma is trivially true. We can assume that we are
in none these two situations, and define $j\in \{2,\ldots, n-1\}$ to be the
smallest possible $i$ such that inequality~\eqref{eq:9} holds. We
will show the conclusion of
Lemma~\ref{lem:connections} holds for $u=u_j$.  We already have 
\begin{equation}
\Pp {A\lr
    L(a,u_j) \cap B(k)} > 1- \varepsilon,
\end{equation}
and inequality \eqref{eq:9} does not hold for $i=j-1$. Once again,
Lemma~\ref{lemFKG} implies that
   \begin{equation}
        \Pp { A\lr L(u_j,b) \cap B(k) }>1-\varepsilon.
  \end{equation}
\end{proof}

\begin{lemma}
\label{lem:perco}
Bernoulli percolation on $\mathcal{G}$ at a parameter
$p>\mathrm{p_c}(\mathcal{G})$ produces almost surely a unique infinite
component. Moreover, any fixed infinite subset of $V$ is intersected
almost surely infinitely many times by the infinite component.
\end{lemma}

The first part of the lemma is standard (see \cite{bk} or
\cite{grim}). The second part stems from the $0\kern 1pt$-$1$ law of Kolmogorov.

\subsection{Geometric constructions}

\label{sec:geom-cons}

In this section, we aim to prove that a set connected to infinity with
high probability also has ``good'' local connections. To formalize this,
we need a few additionnal definitions.  We say that $(a,b,u,v) \in \left(\R^2\right)^4$ is a \emph{good quadruple} if
\begin{enumerate}
\item $u=\frac{a+b}2$,
\item $v\in [-a,b]$ and
\item\label{item:1} $[a,b,-a,-b]$ contains the planar ball of radius $R_S$.
\end{enumerate}
Property~\ref{item:1} ensures that the parallelogramm $[a,b,-a,-b]$ is
not too degenerate. 
\begin{figure}[htbp]
  \centering
  \includegraphics[width=\textwidth/2]{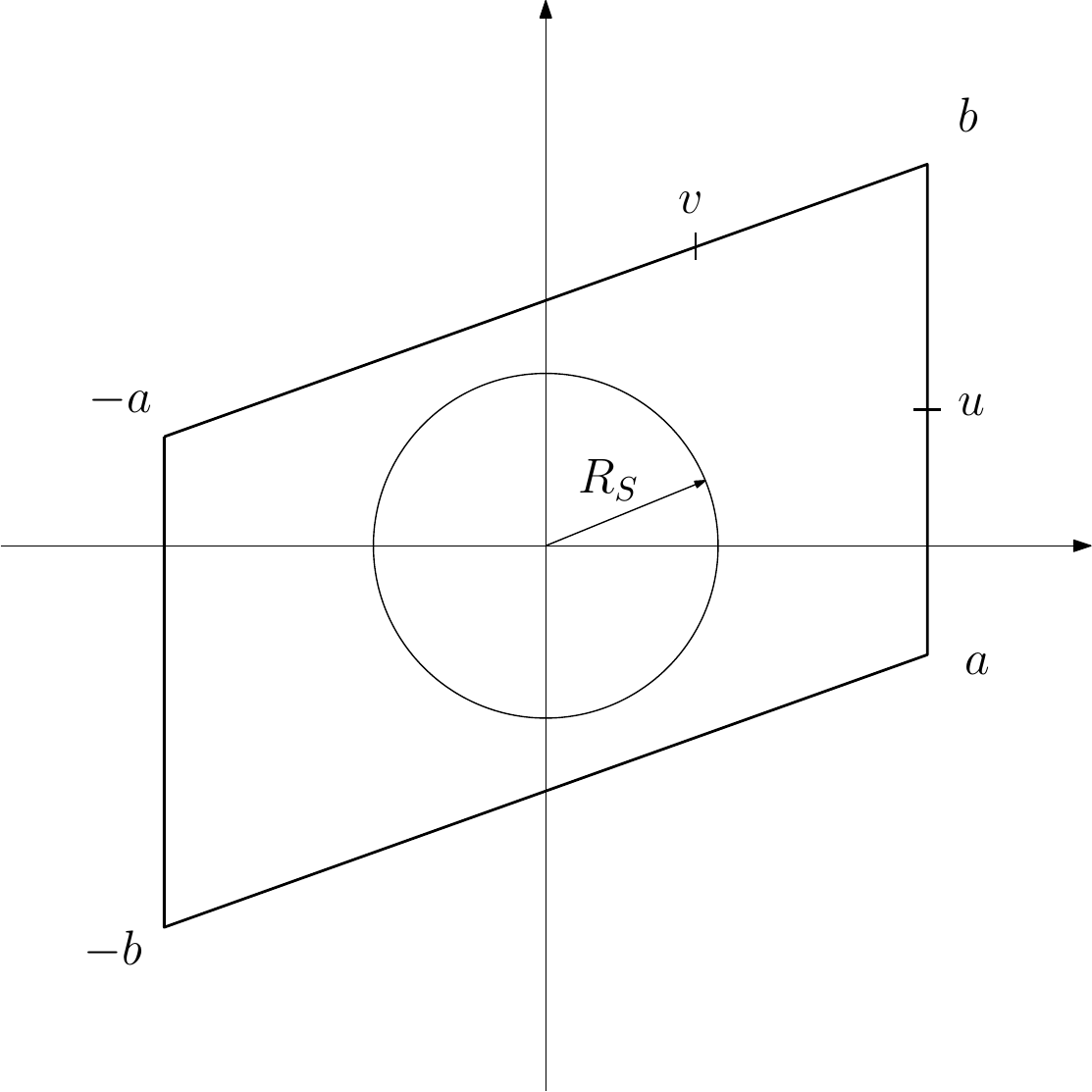}
  \caption{A good quadruple}
  \label{fig:zones}
\end{figure}
To each good quadruple $(a,b,u,v)$, we associate the following four
subsets of the graph~$\cal G$:
\begin{equation}
\mathcal{Z}(a,b,u,v)=
      \left\{ L\left(a,u\right)
      ,L\left(u,b\right),
      L(b,v),L(v,-a) \right\}.
\end{equation}

 \begin{lemma}\label{lem:geometric}
   Let $A$ be a finite subset of $V$ containing $0$ and such that \[
   -A:=\{-x;x\in A\}=A.\] Let $k\ge 1$ be such that $B:=B(k)$ contains $A$. Assume the following relation to hold for some $\varepsilon \in (0,1)$:
   \[
    \Pp{A \lr \infty}>1-\varepsilon^{24}.
   \]
   Then there exists a good quadruple $(a,b,u,v)$ such that for
    any $Z \in \mathcal{Z}(a,b,u,v)$
   \begin{enumerate}[(i)]
   \item $B \cap Z=\emptyset$, \label{item:2}
   \item $ \Pp{A \lr[R(a,b)] Z}>
     1-\varepsilon$.\label{item:3}
   \end{enumerate}
 \end{lemma}

 \begin{proof}
   Let $(n,h,\ell) \in \N \times \R \times \R_+$. Define $a:=(n,h-\ell)$,
   $b:=(n,h+\ell)$ and the three following subsets of $V$ illustrated
   on Figure~\ref{fig:chimney}:
    \begin{align*}
      &C(n,h,\ell): = \gr{[a,b,-a,-b]}\\
      &L \kern -1.7pt R(n,h,\ell):=\gr{[a,b]\cup[-a,-b]} =L(a,b)  \cup L(-a,-b) \\
      &U \kern -3pt D(n,h,\ell):=\gr{[-a,b]\cup[-b,a]} =L(-a,b) \cup
      L(-b,a)
   \end{align*}

   \begin{figure}[htp]
   \centering
   \begin{minipage}[t]{.49\linewidth}
     \raggedleft
     \includegraphics[width=.9\linewidth]{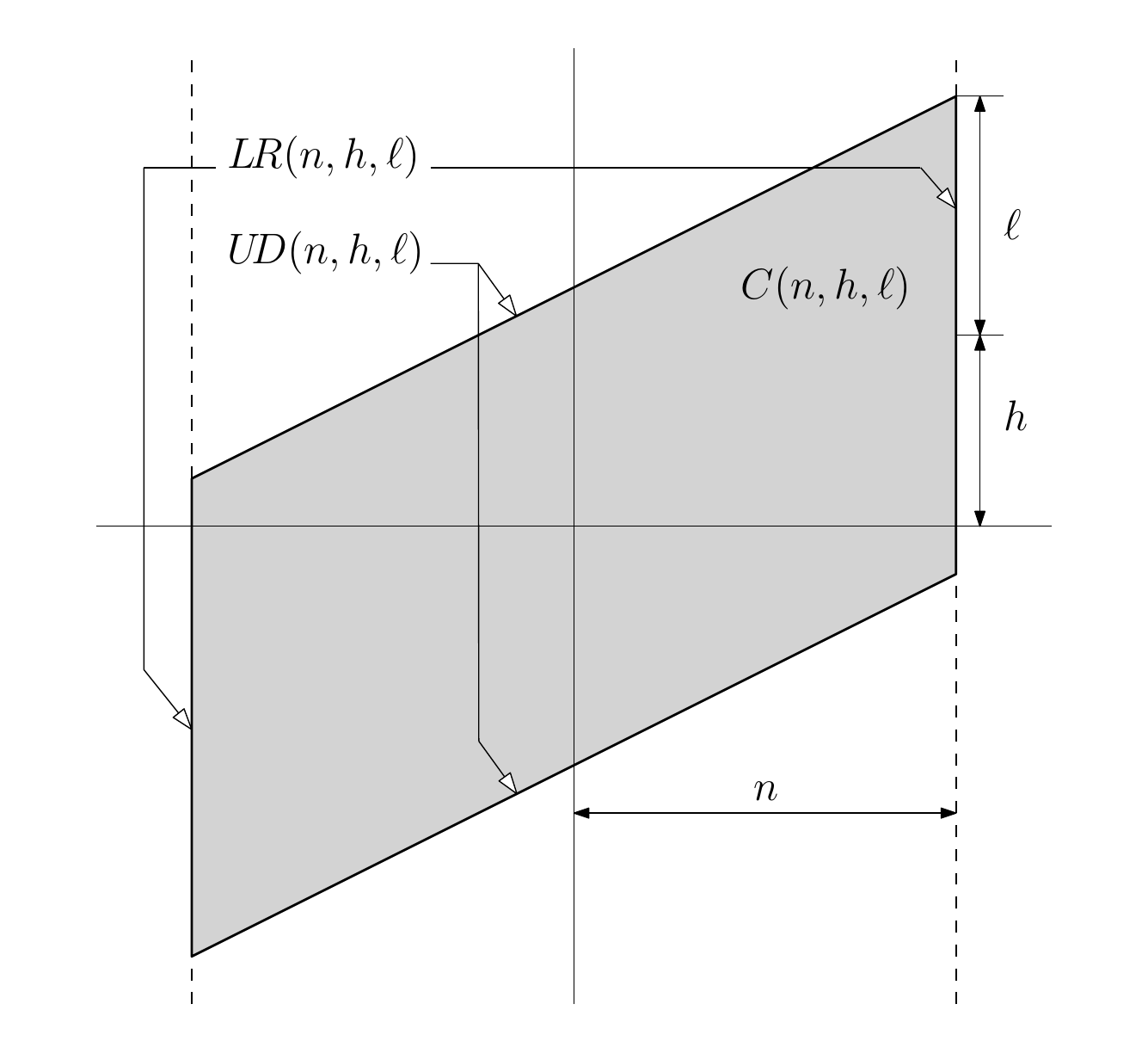} 
     \caption{Pictures of the planar sets defining $C(n,h,\ell)$, $U
       \kern -3pt D(n,h,\ell)$ and $L \kern -1.7pt R(n,h,\ell)$ }
     \label{fig:chimney}
   \end{minipage}
   \begin{minipage}[t]{.49\linewidth}
     \raggedright
     \includegraphics[width=.9\linewidth]{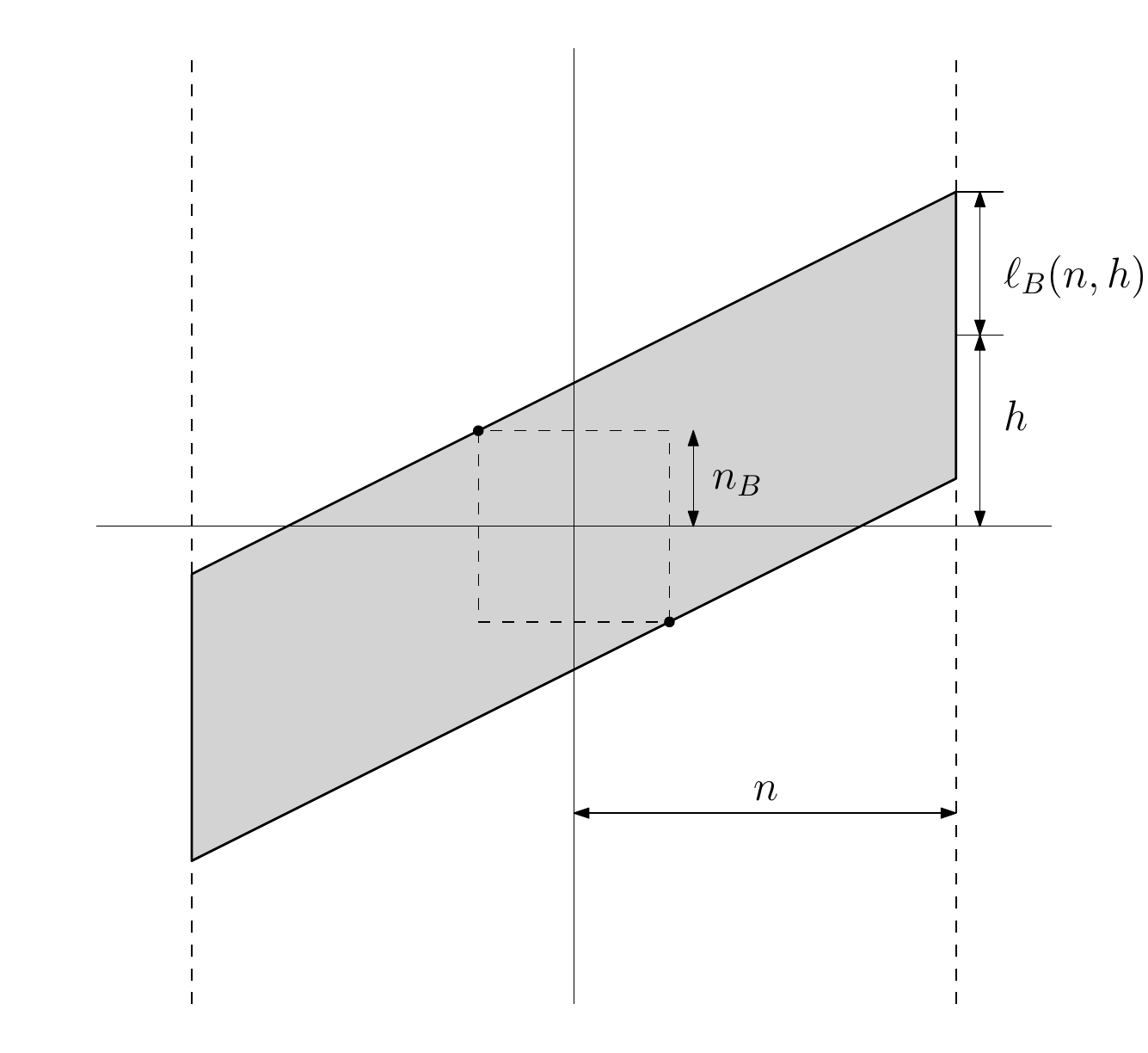}
     \caption{Definition of $\ell_B(n,h)$ }
      \label{fig:deflB}
   \end{minipage}
   \end{figure}

   Let us start by focusing on the geometric
   constraint~{(\ref{item:2})}, which we wish to translate into
   analytic conditions on the triple $(n,h,\ell)$. We fix $n_B$ large
   enough such that 
   \begin{equation}
     \label{eq:10}
      B \cap \gr{\R^2\backslash(-n_B+1,n_B-1)^2} =\emptyset.
   \end{equation}
   This way, any set defined as the image by the function $\sf 
   Graph$ of a planar set in the complement of $(-n_B+1,n_B-1)^2$ will
   not intersect $B$. In particular, defining for $n>n_B$ and
   $h\in\R$
   \begin{equation}
     \ell_B(n,h)=n_B\left(1+\frac {|h|} n\right),
   \end{equation}
   the set $U \kern -3pt D(n,h,\ell)$ does not intersect
$B$ whenever $\ell\ge \ell_B-1$. (See
   Figure~\ref{fig:deflB}.)
   Suppose that $A$ intersects the infinite cluster. By
   Lemma~\ref{lem:perco}, $V\setminus C(n,h,\ell)$ --- which is
   infinite --- intersects the infinite cluster almost surely. Thus
   there exists an open path from $A$ to $V\setminus C(n,h,\ell)$. By
   Lemma~\ref{contdiscrete}, $A$ is connected to $U \kern -3pt
   D(n,h,\ell) \cup L \kern -1.7pt R(n,h,\ell)$ within $C(n,h,\ell)$,
   which gives the following inequality:
   \begin{equation}\label{eq:12}
    \Pp{ \left ( A \lr[C(n,h,\ell)]  L \kern -1.7pt R(n,h,\ell)
       \right ) \cup \left ( A \lr[C(n,h,\ell)] U \kern -3pt D(n,h,\ell)
          \right ) }>1-\varepsilon^{24}.
   \end{equation}

   The strategy of the proof is to work with some sets $C(n,h,\ell)$
   that are balanced in the sense that
   \begin{equation}
\Pp{ A \lr[C(n,h,\ell)] L
     \kern -1.7pt R(n,h,\ell) } \text{ and } \Pp { A \lr[C(n,h,\ell)] U
     \kern -3pt D(n,h,\ell) }
 \end{equation}
 are close, and conclude with Lemma~\ref{lemFKG}. We shall now prove
 two facts, which ensure that the inequality between the two
 afore-mentioned probabilities reverses for some $\ell$ between
 $\ell_B(n,h)$ and infinity.
   
 \begin{fact}
   \label{lemBinC}
   There exists $n> n_B$ such that, for all $h \in \R$, when $\ell =
   \ell_B(n,h)$
   \[
   \Pp{A \lr[C(n,h,\ell)] L \kern -1.7pt R(n,h,\ell)} < \Pp{A
     \lr[C(n,h,\ell)] U \kern -3pt D(n,h,\ell)}.
   \]
 \end{fact}
 \begin{proof}[Proof of fact~\ref{lemBinC}]
   For $n> n_B+R_S$, define the following sets, illustrated on
   Figure~\ref{fig:butterfly}:
   \begin{align*}
     &X= \gr{\left((-\infty,n_B)\times \R\right) \cup
       \left(\R\times [-n_B,\infty)\right)}\\
     &\partial X=\gr{\left( \{n_B\}\times (-\infty,-n_B] \right) \cup
       \left([n_B,\infty)\times\{-n_B\}\right)} \\
     &X_n=\gr{\left([-n,n_B)\times \R\right) \cup
       \left([-n,n]\times [-n_B,\infty)\right)}\\
     &\partial_1{X_n}=\gr{\{-n\}\times\R\cup\{n\}\times[-n_B,\infty)}\\
     &\partial_2{X_n}=\gr{ \{n_B\}\times (-\infty,-n_B]\cup
       [n_B,n]\times\{-n_B\}}
     \end{align*}
     \begin{figure}[htp]
       \centering
       \includegraphics[width=7cm]{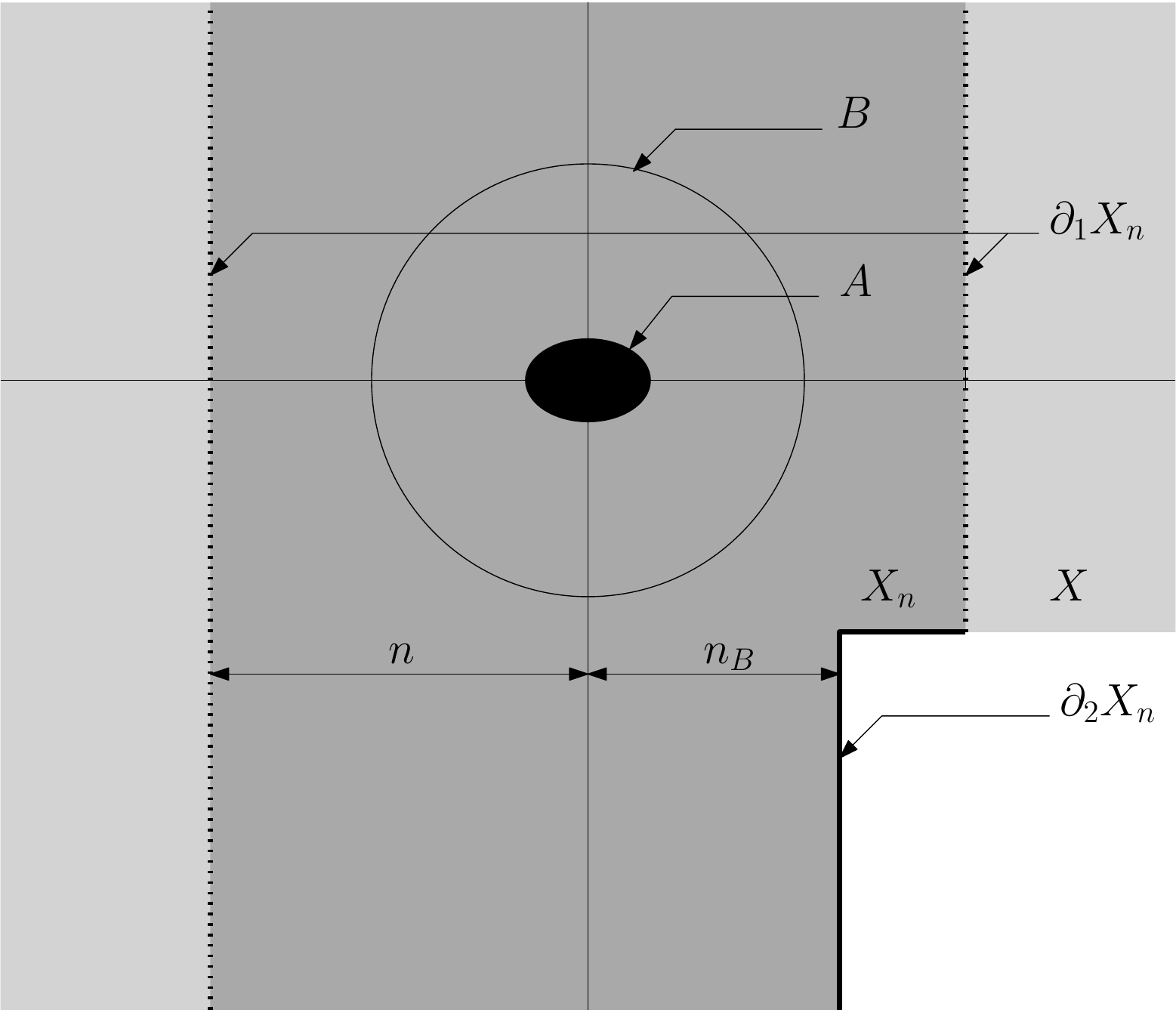}
       \caption{Planar pictures corresponding to $X$,
         $X_n$, $\partial_1X_n$ and $\partial_2X_n$}
       \label{fig:butterfly}
     \end{figure}
     Since the sequence of events $\left ( A
       \lr[X_n] \partial_1{X_n} \right )_{n> n_B + R_S}$ is
     decreasing, we have
    \begin{align}
     \lim_{n \to \infty} \Pp{ A
       \lr[X_n] \partial_1{X_n}}&=\Pp{\bigcap_{n > n_B + R_S} \left
         ( A \lr[X_n] \partial_1{X_n}\right )} \notag\\
     & \leq \Pp{ A \lr[X] \infty } \notag\\
     & = \Pp{ \left ( A \lr[X] \infty \right) \cap \left(  A
         \lr[X] \partial X\right)  }.\label{eq:13}
    \end{align}
    (The last equality results from the fact that the infinite set $V
    \setminus X$ intersects the infinite cluster almost surely.)
    \vspace{0.2 cm}

     The sequence $\left  ( A \lr[X_n] \partial_2{X_n} \right
    )_{n > n_B + R_S}$ is increasing, hence we have
    \begin{align}
     \lim_{n \to \infty} \Pp{ A
       \lr[X_n] \partial_2{X_n}}&=\Pp{\bigcup_{n > n_B + R_S} \left
         ( A \lr[X_n] \partial_2{X_n}\right )} \notag\\
     & = \Pp{ A \lr \partial X }. \label{eq:14}
    \end{align}
    Since $p\in (0,1)$ and $A$ is finite, the probability that $A$ is connected
    to $\partial X$ but intersects only finite clusters is positive.
    Thus the following strict inequality holds
    \begin{equation}
      \Pp{ \left ( A \lr[X] \infty \right) \cap \left(  A
         \lr[X] \partial X \right)  } < \Pp{ A \lr \partial X }. \label{eq:15}
    \end{equation}
     From \eqref{eq:13}, \eqref{eq:14} and \eqref{eq:15}, we can pick
     $n_1 > n_B+R_S$ large enough such that,
     for all $n \geq n_1$,
     \begin{equation}
              \Pp{A \lr[X_n] \partial_1{X_n}} < \Pp{ A
         \lr[X_n] \partial_2{X_n}}.
     \end{equation}
     Fix $n \geq n_1$ and  $h \geq 0$, then define $\ell=\ell_B(n,h)$.
     For these parameters, we have $ A\subset C(n,h,\ell) \subset
         X_n$ and $L \kern -1.7pt R(n,h,\ell)
         \subset \partial_1X_n $, which  gives
     \begin{align*}
       \Pp{A \lr[C(n,h,\ell)] L \kern -1.7pt R(n,h,\ell)} &
       \leq
       \Pp{A \lr[X_n] \partial_1X_n} \\
       &< \Pp{A \lr[X_n] \partial_2X_n} \\
       & \leq
       \Pp{A \lr[C(n,h,\ell)] U \kern -3pt D(n,h,\ell)}.
     \end{align*}
     The last inequality follows from the observation that each path
     connecting $A$ to $\partial_2X_n $ inside $X_n$ has to cross $U
     \kern -3pt D(n,h,\ell)$.

     The computation above shows that the following srict inequality
     holds for $n \geq n_1$, $h \geq 0$, and $\ell=\ell_B(n,h)$
     \begin{equation}\label{eq:16}
     \Pp{A \lr[C(n,h,\ell)] L \kern -1.7pt R(n,h,\ell)} < \Pp{A
       \lr[C(n,h,\ell)]  U \kern -3pt D(n,h,\ell)}.
     \end{equation}
      In the same way, we find $n_2$ such that for all $n\geq n_2$ and $h
     \leq 0$, equation~\eqref{eq:16} holds for $\ell=\ell_B(n,h)$. Taking
     $n=\max(n_1,n_2)$ ends the proof of the fact.
   \end{proof}

   In the rest of the proof, we fix $n$ as in the previous fact.
   For $ h \in \R$, define
   \begin{equation}
     \ell_{\mathrm{eq}}(h) =
     \begin{array}[t]{ll}
       \sup \Big \{ \ell\geq \ell_B(n,h)-1 \: : \: &\Pp{
         A \lr[C(n,h,\ell)] U \kern -3pt D(n,h,\ell)} \\
        &\geq \Pp{  A \lr[C(n,h,\ell)] L \kern -1.7pt R(n,h,\ell)} \Big \}.
     \end{array}\\
   \end{equation}
   \begin{fact}\label{lemBinCbis}
     For all $h \in \R$, the quantity $\ell_{\mathrm{eq}}(h)$ is finite.
   \end{fact}
   \begin{proof}[Proof of fact~\ref{lemBinCbis}]
     We fix $h\in \R$ and use the same
     technique as developed in the proof of the fact~\ref{lemBinC}.
     Define
     \begin{align*}
       &Y= \gr{[-n,n]\times \R}\\
        &\partial Y= \gr{\{-n,n\}\times \R}
     \end{align*}
     In the same way we proved equations \eqref{eq:13} and
     \eqref{eq:14}, we have here
     \begin{align}
       \lim_{\ell \to \infty} \Pp{A \lr[C(n,h,\ell)] U \kern -3pt
         D(n,h,\ell)} &=
       \Pp{A \lr[Y] \infty}\\
       \lim_{\ell \to \infty} \Pp{A \lr[C(n,h,\ell)] L \kern -1.7pt
         R(n,h,\ell)} &= \Pp{A \lr \partial Y}
     \end{align}
     Thus, we can find a finite $\ell$ large enough such that
     \[
     \Pp{A \lr[C(n,h,\ell)] U \kern -3pt D(n,h,\ell)} <
     \Pp{A \lr[C(n,h,\ell)] L \kern -1.7pt R(n,h,\ell)}.
     \]
   \end{proof}

   The quantity $\ell_{\mathrm{eq}}$ plays a central role in our proof,
   linking geometric and probabilistic estimates.
   We can apply Lemma~\ref{lemFKG} with the two events appearing in 
   inequality~\eqref{eq:12}, to obtain the following alternative:
   \begin{subequations}
     \begin{align}
       \text{If }\ell < \ell_{\mathrm{eq}}(h) \text{,}& \quad
       \text{then } \Pp{ A
         \lr[C(n,h,\ell)] U \kern -3pt D(n,h,\ell)}> 1-\varepsilon^{12}.\label{eq:17}\\
       \text{If } \ell > \ell_{\mathrm{eq}}(h) \text{,}& \quad
       \text{then } \Pp{ A \lr[C(n,h,\ell)] L \kern -1.7pt
         R(n,h,\ell)}> 1-\varepsilon^{12}.\label{eq:18}
     \end{align}
   \end{subequations}
  
   Fix $(h_{\mathrm{opt}},\ell_0)\in \R\times\R_+$ such that
   \begin{equation}
     \label{eq:19}
     \ell_{\mathrm{eq}}(h_{\mathrm{opt}})< \ell_0 < \inf_{h\in \R} \left ( \ell_{\mathrm{eq}}(h) \right )+\tfrac16 .
   \end{equation}
   With such notation, we derive from \eqref{eq:18}
   \begin{equation}
     \Pp{A\lr[C(n,h_{\mathrm{opt}},\ell_0)]L \kern -1.7pt
       R(n,h_{\mathrm{opt}},\ell_0)}> 1-\varepsilon^{12}.
   \end{equation}
   Another application of Lemma~\ref{lemFKG} ensures then the
   existence of a real number $h_0$ of the form
   $h_0=h_{\mathrm{opt}}+\sigma \ell_0/3$ (for $\sigma \in
   \{-2,0,+2\}$) such that

\begin{equation}
  \Pp{A \lr[C(n,h_{\mathrm{opt}},\ell_0)] L \kern -1.7pt R(n,h_0,\ell_0/3)} > 1-\varepsilon^4.
\end{equation}
Recall that $L \kern -1.7pt R(n,h_0,\ell_0/3)=L(a_0,b_0)\cup
L(-a_0,-b_0)$ with $a_0=(n,h_0-\ell_0/3)$ and $b_0=(n,h_0+\ell_0/3)$.
By symmetry, the set $A$ is connected inside $C(n,h_0,\ell_0/3)$ to
$L(a_0,b_0)$ and to $L(-a_0,-b_0)$ with equal probabilities. Applying
again Lemma~\ref{lemFKG} gives
\begin{equation}
   \Pp{ A \lr[C(n,h_{\mathrm{opt}},\ell_0)] L(a_0,b_0)}> 1-\varepsilon^2.
\end{equation}
Then, use Lemma~\ref{lem:connections} to split $L(a_0,b_0)$ into two
parts that both have a good probability to be connected to $A$: we can
pick $u=(n,h)\in [a_0,b_0]$ such that both \[\Pp{ A
  \lr[C(n,h_{\mathrm{opt}},\ell_0)] L(a_0,u)}\text{ and }\Pp{ A
  \lr[C(n,h_{\mathrm{opt}},\ell_0)] L(u,b_0)}\] exceed
$1-\varepsilon$. Finally, pick $\ell$ such that
$\ell_{\mathrm{eq}}(h)-\tfrac16<\ell<\ell_{\mathrm{eq}}(h)$.
Define
$a=u+(0,-\ell)$ and $b=u+(0,\ell)$. In particular, we have
$u=\frac{a+b}{2}$. 
Our choice of $\ell_0$ (see equation~\eqref{eq:19})
implies that $\ell > \ell_0-\tfrac13\ge  \tfrac23 \ell_0$, and the following inclusions hold:
\begin{align}
  & L(a_0,u)\subset L(a,u)\\
  & L(u,b_0)\subset L(u,b)\\
  & C(n,h_{\mathrm{opt}},\ell_0) \subset R(a,b)
\end{align}
These three inclusions together with
the estimates above conclude the point~(\ref{item:3}) of
Lemma~\ref{lem:geometric} for $Z=  L(a,u)$ and $Z=L(u,b)$.

Now, let us construct a suitable vector $v\in[-a,b]$ such that the
point~(\ref{item:3}) of Lemma~\ref{lem:geometric} is
verified for $Z=L(-a,v)$ and $Z=L(v,b)$. Equation  \eqref{eq:17} implies that
\begin{equation}
  \Pp{ A
    \lr[C(n,h,\ell)] U \kern -3pt D(n,h,\ell)}> 1-\varepsilon^{12}.
\end{equation}

As above, using $U \kern -3pt D(n,h,\ell) = L(-a,b) \cup L(-b,a)$,
symmetries and Lemma~\ref{lemFKG}, we obtain
\begin{equation}
  \Pp{ A \lr[C(n,h,\ell)] L(-a,b)}> 1-\varepsilon^{6}.
\end{equation}
By Lemma~\ref{lem:connections}, we can pick $v\in[-a,b]$ such that the
following estimate holds for $Z= L(-a,v), \,L(v,b)$:
\begin{equation}
    \Pp{ A \lr[C(n,h,\ell)] Z}> 1-\varepsilon^{3} \ge 1-\varepsilon.
\end{equation}

 It remains to verify
the point~(\ref{item:2}).
 For  $Z=L(a,u),L(u,b)$, it follows from $n>n_B$
 and the definition of $n_B$, see equation~\eqref{eq:10}. For
 $Z=L(-a,v), L(v,b)$, it follows from $\ell > \ell_B(n,h)-1$ (see
 Fact~\ref{lemBinC}) and the definition of $ \ell_B(n,h)$.
\end{proof}

\subsection{Construction of Good Blocks}

\label{sec:good-blocks}

In this section, we will define a finite block together with a local
event that ``characterize'' supercritical percolation --- in the sense
that the event happening on this block with high probability will
guarantee supercriticality. This block will be used in
section~\ref{sec:renormalization} for a coarse graining argument.

In Grimmett and Marstrand's proof of
Theorem~\ref{sec:slabs-conv-symm}, the coarse graining argument uses
``seeds'' (big balls, all the edges of which are open) in order to
propagate an infinite cluster from local connections. More precisely,
they define an exploration process of the infinite cluster: at each
step, the exploration is succesful if it creates a new seed in a
suitable place, from which the  process can iterate. If the
probability of success at each step is large enough, then, with
positive probability, the exploration process does not stop and an
infinite cluster is created.

In their proof, the seeds grow in the unexplored region. Since we
cannot control this region, we use the explored region to produce
seeds instead. Formally, long finite self-avoiding paths will play the
role of the seeds in the proof of Grimmett and Marstrand. The idea is
the following: if a point is reached at some step of the exploration
process, it must be connected to a long self-avoiding path, which is
enough to iterate the process.

\begin{lemma}
  \label{lem:path}
  For all $\varepsilon>0$, there exists $m\in\N$ such that, for any fixed
  self-avoiding path $\gamma$ of length $m$,
  \begin{equation}
    \Pp{ \gamma \lr{} \infty} >1-\varepsilon.
  \end{equation}
\end{lemma}

\begin{proof}
  By translation invariance we can restrict ourselves to self-avoiding
  paths starting at the origin $0$. Fix $\varepsilon>0$. For all $k\in
  \N$ we consider one self-avoiding path $\gamma^{(k)}$ starting at
  the origin that minimizes the probability to intersect the infinite
  cluster among all the self-avoiding paths of length $k$:
  \begin{equation}
    \Pp{\gamma^{(k)}\lr\infty}=\min_{\gamma :\,\mathrm{length}(\gamma)=k}\Pp{\gamma\lr\infty}.
  \end{equation}
  By diagonal extraction, we can consider an infinite self-avoiding
  path $\gamma^{(\infty)}$ such that, for any $k_0\in\N$,
  $\left(\gamma^{(\infty)}_0,\gamma^{(\infty)}_1,\ldots,\gamma^{(\infty)}_{k_0}
  \right )$ is the beginning of infinitely many $\gamma^{(k)}$'s. By
  Lemma~\ref{lem:perco}, $\gamma^{(\infty)}$ intersects almost surely
  the infinite cluster of a $p$-percolation. Thus, there exists an
  integer $k_0$ such that
  \begin{equation}
      \Pp{\left\{\gamma^{(\infty)}_0,\gamma^{(\infty)}_1,\ldots,\gamma^{(\infty)}
        _{k_0}\right\}\lr{}\infty} > 1-\varepsilon.
  \end{equation}
  Finally, there exists $m$ such that $\gamma_m$ begins with the
  sequence $$(\gamma^{(\infty)}_0,\gamma^{(\infty)}_1,\ldots,\gamma^{(\infty)}_{k_0}),$$
  thus it intersects the infinite cluster of a $p$-percolation with
  probability exceeding $1-\varepsilon$. By choice of $\gamma^{(m)}$,
  it holds for any other self-avoiding path $\gamma$ of length $m$
  that
  \begin{equation}
    \Pp{\gamma \lr \infty} >1-\varepsilon.
  \end{equation}
\end{proof}

We will focus on paths that start close to the origin. Let us define
$\mathcal{S}(m)$ to be the set of self-avoiding paths of length
$m$ that start in $B(1)$.

\begin{lemma}
  \label{lemgeom1}
  For any $\eta>0$, there exist two integers $m, N \in \N$ and a good
  quadruple $(a,b,u,v)$ such that
  \begin{equation}
    \forall  \gamma \in \mathcal{S}(m),\,\forall Z \in \mathcal{Z}(a,b,u,v)\quad
    \Pp{\gamma \lr[R(a,b)\cap B(N)] Z\cap B(N)}>1-3\eta.
  \end{equation}
\end{lemma}

\begin{proof}
  By Lemma~\ref{lem:path}, we can pick $m$ such that any self-avoiding
  path $\gamma \in \mathcal S(m)$ verifies
  \begin{equation}
    \Pp{ \gamma \lr \infty } > 1-\eta.
  \end{equation}
  Pick $k\geq m+1$ such that
  \begin{equation}
    \Pp{ B(k)\lr \infty } > 1- \eta^{24}.
  \end{equation}
  The number of disjoint clusters (for the configuration restricted to
  $B(n+1)$) connecting $B(k)$ to $B(n)^c$ converges when
  $n$ tends to infinity to the number of infinite clusters
  intersecting $B(k)$. The infinite cluster being unique, we can pick
  $n$ such that
  \begin{equation}
    \Pp{ B(k) \lr[!B(n+1)!] B(n)^c } > 1-\eta.\label{eq:21}
  \end{equation}
  Applying Lemma~\ref{lem:geometric} with $A=B(k)$ and $B=B(n+1)$
  provides a good quadruple $(a,b,u,v)$ such that the following two
  properties hold for any $Z \in \mathcal{Z}(a,b,u,v)$:
  \begin{enumerate}[(i)]
  \item $B(n+1) \cap Z =\emptyset$, \label{item:4}
  \item $ \Pp{B(k) \lr[R(a,b)] Z} > 1-\eta$\label{item:5}.
  \end{enumerate}
    Note that condition~(\ref{item:4}) implies in
  particular that $B(n+1)$ is a subset of $R(a,b)$.
  Equation~\eqref{eq:21} provides with high probability a ``uniqueness
  zone'' between $B(k)$ and $B(n)^c$: any pair of open paths
  crossing this region must be connected inside $B(n+1)$. In
  particular, when $\gamma$ is connected to infinity, and $B(k)$ is
  connected to $Z$ inside $R(a,b)$, this ``uniqueness zone'' ensures
  that $\gamma$ is connected to $Z$ by an open path lying inside
  $R(a,b)$:
\begin{align}
  &\Pp{\gamma \lr[R(a,b)] Z   }\\
  &\quad
    \geq \Pp{
              \left \{ \gamma \lr \infty \right \}
        \cap \left \{B(k) \lr[!B(n+1)!] B(n)^c \right \}
        \cap \left \{B(k) \lr[R(a,b)] Z \right \}  } \\
  &\quad >1-3\eta.
\end{align}
The identity
\begin{equation}
  \Pp{\gamma \lr[R(a,b)] Z}=\lim_{N\to \infty}\Pp{\gamma \lr[R(a,b)\cap B(N)] Z\cap B(N)}
\end{equation}
concludes the proof of Lemma~\ref{lemgeom1}.
\end{proof}

\subsection{Construction of a finite-size criterion}
\label{sec:stab-lemma-under}

In this section, we give a precise definition of the finite-size criterion
$\mathscr{FC}(p,N,\eta)$ used in lemmas~\ref{lem:finite-criterion} and
\ref{lem:renormalization}. Its construction is based on
Lemma~\ref{lemgeom1}. 

Recall that, up to now, we worked with a fixed orthonormal basis
$\mathbf e$, which was hidden in the definition of
$\mathsf{Graph}=\mathsf{Graph}_{\mathbf{e}}$, see
equation~\eqref{eq:5}. In order to perform the coarse graining
argument in any marked group $G^\bul/\Lambda$ close to $G^\bul$, we
will need to have the conclusion of Lemma~\ref{lemgeom1} for all the
orthonormal bases.

Denote by $\mathfrak B$ the set of the orthonormal basis of
$\R^r$. It is a compact subset of $\R^{r\times r}$. If we
fix $X\subset \R\textsc{}^2$, a positive integer $N$ and $\mathbf{e}\in\mathfrak B$
then the following inclusion holds for any orthonormal basis
$\mathbf{f}$ close enough to $\mathbf e$ in $\mathfrak B$:
\begin{equation}
  \label{eq:22}
   \gre[e]{X} \cap B(N)  \subset \left(\gre[f]{X}+B(1)\right)\cap B(N).
\end{equation}
We define $\mathcal{N}({\mathbf e},N)\subset \mathfrak{B}$ to be the
neighbourhood of $\mathbf e$ formed by the orthonormal bases $\mathbf f$ for
which the inclusion above holds. A slight
modification of the orthonormal basis in Lemma~\ref{lemgeom1} keeps its
conclusion with the same integer $N$ and the same vectors $a,b,u,v$, but
with
\begin{itemize}
\item $Z+B(1)$ in place of $Z$
\item  and $R(a,b)+B(1)$ instead of $R(a,b)$.
\end{itemize}

In order to state this result properly, let us define:
\begin{align}
  &\mathcal Z_{N,\mathbf e}(a,b,u,v) =\{(Z+B(1))\cap B(N):\:Z\in
  \mathcal
  Z_{\mathbf e}(a,b,u,v)\};\\
  &R_{N,\mathbf e}(a,b)=(R(a,b)+B(1))\cap B(N).
\end{align}
Note that we add the subscript $\mathbf e$ here to insist on the
dependence in the basis $\mathbf e$. This dependence was implicit for
the sets $Z$ and $R(a,b)$ which were defined via the function
$\mathsf{Graph}$.

We are ready to define the finite size criterion
$\mathscr{FC}(p,N,\eta)$ that appears in
lemmas~\ref{lem:finite-criterion} and \ref{lem:renormalization}.

  \begin{definition}
    Let $N\ge 1$ and $\eta>0$. We say that the finite size criterion
    $\mathscr{FC}(p,N,\eta)$ is satisfied if for any $\mathbf e \in
    \mathfrak B$, there exist $m\ge 1$ and a good quadruple $(a,b,u,v)$
    such that:
   \begin{equation}
   \label{eq:20}
     \forall \gamma\in\mathcal{S}(m),\,\forall
      Z\in\mathcal{Z}_{N,\mathbf e}(a,b,u,v),\quad \Pp { \gamma
        \lr[R_{N,\mathbf e}(a,b)] Z}> 1-\eta .
    \end{equation}
  \end{definition}

\begin{proof}[\bf Proof of Lemma~\ref{lem:finite-criterion}]
  Let $\eta>0$. Given $\mathbf{e}$ an orthonormal basis,
  Lemma~\ref{lemgeom1} provides $m_{\mathbf e}, N_{\mathbf e} \in \N$,
  and a good quadruple $(a_{\mathbf e}, b_{\mathbf e},u_{\mathbf{e}}, v_{\mathbf
    e})$ such that the following holds (we omit the subscript for the
  parameters $m,a,b,u,v$):
 \begin{equation}
    \forall \gamma\in\mathcal{S}(m),\,\forall
   Z\in\mathcal{Z}_{\mathbf e}(a,b,u,v),\quad \Pp { \gamma \lr[R_{\mathbf e}(a,b) \cap B(N_{\mathbf
       e})] Z\cap B(N_{\mathbf e})}> 1-\eta.
\end{equation}
For any $\mathbf f \in \mathcal N(\mathbf e,N_{\mathbf e})$, we can use
inclusion~\eqref{eq:22} to derive from the estimate above that
for all $\gamma\in\mathcal{S}(m)$ and $Z \in\mathcal{Z}_{\mathbf f}(a,b,u,v)$,
\begin{equation}
   \Pp { \gamma \lr[(R_{\mathbf f}(a,b)+B(1)) \cap
     B(N_{\mathbf e})]
     (Z+B(1))\cap B(N_{\mathbf e})}> 1-\eta.
\end{equation}
 By compactness of $\mathcal B$, we
  can find a finite subset $\mathcal F\subset \mathcal B$ of bases such that
  \begin{equation}
      \mathcal B = \bigcup_{\mathbf e\in \mathcal F}
    \mathcal{N}({\mathbf e},N_{\mathbf e}).
  \end{equation}
 For  $\displaystyle N:=\max_{\mathbf{e}\in \mathcal B_f}N_{\mathbf
    e}$, the finite-size criterion $ \mathscr{FC}(p,N,\eta)$ is satisfied.

\end{proof}

\section{Proof of Lemma~\ref{lem:renormalization}}
\label{sec:renormalization}
Through the entire section, we fix:
\begin{itemize}
\item[-] $\grp^\bul \in \tilde\Grp$ a marked abelian group of rank
greater than two,
\item[-] $p\in (\mathrm{p_c}^{\!\!\!\bul}(G^\bul),1)$,
\item[-] $\delta>0$.
\end{itemize}
Let $\mathcal{G}=(V,E)$ denote the Cayley graph
associated to $\grp^\bul$.

\subsection{Hypotheses and notation}
Let us start by an observation that follows from the
definition of good quadruple at the beginning of
section~\ref{sec:geom-cons}: there exists an absolute constant $\kappa$ such that
for any good quadruple $(a,b,u,v)$ and any $w\in \R^2$,
\begin{equation}
    \mathrm{Card} \left\{z\in \Z^2 \::\: w + z_1 u + z_2 v\in [5a,5b,-5a,-5b]
    \right\} \le \kappa.
\end{equation}
We fix $\kappa$ as above and choose $\eta>0$ such that
\begin{equation}
  \label{eq:23}
  p_0:=\sup_{t\in\N}\left\{1-
    (1-\delta/\kappa)^t+\eta(1-p)^{-t}\right\}>\mathrm{\mathrm{p_c}^{\!\!\!
      site}}(\Z^2).
\end{equation}
We will prove that this choice of $\eta$ provides the conclusion of
Lemma~\ref{lem:renormalization}. 
We assume that $G^\bul$ satisfies $\mathscr{FC}(p,N,\eta)$ for some
positive integer $N$ (which will be fixed throughout this section). 
Let us consider a marked abelian group $H^\bul=G^\bul/\Lambda$ of rank
at least $2$ and such that
\begin{equation}
 \Lambda \cap B(2N+1)= \{0\}.
\end{equation}
(Notice that such $H^\bul$'s form a neighbourhood of $G^\bul$ in
$\tilde{\cal G}$ by Proposition~\ref{prop:conv-seq}.) 
Under these hypotheses, we will prove that
$\mathrm{p_c}(H^\bul)<p+\delta$, providing the conclusion of
Lemma~\ref{lem:renormalization}.

The Cayley graph of $H^\bul=\grp^\bul/ \Lambda$ is denoted by
$\overline{\mathcal{G}}=(\overline V,\overline E)$. For $x\in V$, we
write $\bar x$ for the image of $x$ by the quotient map $G\to G/\Lambda$.
This quotient map naturally extends  to subsets of $V$ and we write
$\overline{A}$ for the image of a set $A\subset V$.

\subsection{Sketch of proof}
\label{sec:sketch-proof}
Under the hypotheses above, we show that percolation occurs in
$\overline{\mathcal G}$ at parameter $p+\delta$. The proof goes as follows.
\begin{description}
\item[\bf Step 1: Geometric construction.] We will construct a
  renormalized graph, that is a family of big boxes (living in
  $\overline{\mathcal G}$) arranged as a square lattice. In
  particular, there will be a notion of neighbour boxes. The occurence
  of the finite-size criterion $\mathscr{FC}(p,N,\eta)$ will imply
  good connection probabilities between neighbouring boxes. This is the
  object of Lemma~\ref{sec:probabilist-setting}.
\item[\bf Step 2: Construction of an infinite cluster.] The
  renormalized graph built in the first step will allow us to couple a
  $(p+\delta)$-percolation on $\overline{\mathcal G}$ with a
  percolation on $\Z^2$ in such a way that the existence of an
  infinite component in $\Z^2$ would imply an infinite component in
  $\overline{\mathcal G}$. This event will happen with positive
  probability. The introduction of the parameter $\delta$ will allow
  us to apply a ``sprinkling'' technique in the coupling argument
  developed in the proof of Lemma~\ref{sec:probabilist-setting-1}.
\end{description}

\subsection{Geometric setting: boxes and corridors}

 Since $\Lambda$ has corank at least $2$, we can fix an orthonormal
  basis $\mathbf e\in \mathcal B$ such that
  \begin{equation}
    \label{eq:24}
    \Lambda\subset \mathrm{Ker}\left(\pi_{\mathbf e}\right) \times T.
  \end{equation}
  Condition~\eqref{eq:24} ensures that sets defined in $\mathcal G$
  via the function $\mathsf{Graph}_{\mathbf e}$ have a suitable image
  in the quotient $\overline{\cal G} $. More precisely, for any $x\in
  V$ and any planar set $X\subset \R^2$, we have
  \begin{equation}
    x\in \gre{X} \iff \overline x \in \overline{\gre{X}}\label{eq:25}.
  \end{equation}
  According to $\mathscr{FC}({p,N,\eta})$, there exists $m<N$
  and a good quadruple $(a,b,u,v)$ such that
 \begin{equation}
   \forall \gamma\in\mathcal{S}(m),\,\forall Z\in\mathcal{Z}_{N,\mathbf
     e}(a,b,u,v),\quad \Pp { \gamma \lr[R_{N,\mathbf e}(a,b)]
     Z }> 1-\eta .
 \end{equation}
 We introduce here some subsets of $\overline{\mathcal{G}}$, that will
 play the role of vertices and edges in the renormalized graph.
 
 \emph{Box}. For $z$ in $\Z^2$, define
 \begin{equation}
   B_z:=\overline{\gr{z_1 u +z_2 v + [a,b,-a,-b] }}.
 \end{equation}
 When $z$ and $z'$ are neigbours in $\Z^2$ for the standard graph
 structure, we write $z\sim z'$. In this case, we say that the two boxes
 $B_z$ and $B_{z'}$ are \emph{neighbours}.
 
 \emph{Corridor}. For $z$ in $\Z^2$, define
 \begin{equation}
   C_z:=\overline{\gr{z_1 u +z_2 v + [4a,4b,-4a,-4b]}}.
 \end{equation}
 We will explore the cluster of the origin in $\overline{\mathcal G}$.
 If the cluster reaches a box $B_{z}$, we will try to spread it to the
 neighbouring boxes ($B_{z'}$ for $z'\sim z$) by creating paths that
 lie in their respective corridors $C_{z'}$. For this strategy to
 work, we need the boxes to have good connection probabilities and the
 corridors to be ``sufficiently disjoint'': if the exploration is
 guaranted to visit each corridor at most $\kappa+1$ times, then we do
 need more than $\kappa$ ``sprinkling operations''. These two
 properties are formalized by the following two lemmas.
\begin{lemma}
  \label{sec:geom-sett-boxes}
  For all $\bar x\in V$,
  \begin{equation}
    \label{eq:26}
    \mathrm{Card} \{z\in \Z^2 \:/\: \bar x \in C_z\}\leq \kappa.
  \end{equation}
\end{lemma}
\begin{proof}
  By choice of the basis, equivalence~\eqref{eq:25} holds and implies,
  for any $z=(z_1,z_2)\in \Z^2$,
  \begin{equation}
    \bar x \in C_z \iff x \in \gre{z_1 u +z_2 v +[4a,4b,-4a,-4b]}\}
  \end{equation}
  By the last condition defining a good quadruple,
  \begin{equation}
      \bar x \in C_z \implies \pi(x)\in z_1 u +z_2 v +[5a,5b,-5a,-5b]
  \end{equation}
  The choice of $\kappa$ at the beginning of the section (see
  equation~\eqref{eq:26}) concludes the proof.
\end{proof}

\begin{lemma}
  \label{sec:probabilist-setting}
  For any couple of neighbouring boxes $(B_z,B_{z'})$,
  \begin{equation}
    \label{eq:27}
   \forall \bar x \in B_z,\forall \gamma\in \mathcal{S}(m) \quad \Pp{\bar x+\overline{\gamma}
     \lr[C_{z'}] B_{z'}+\overline{B(1)}}>1-\eta.
 \end{equation}
\end{lemma}
\begin{proof}
  We assume that $z'=z+(0,1)$. The cases of $z+(1,0)$, $z+(0,-1)$ 
  and  $z+(-1,0)$ are treated the same way.

  The assumption $\Lambda \cap B(2N+1)=\{0\}$ implies that 
  $\overline{R_{N,\mathbf e}(a,b)}$ is isomorphic (as a graph) to 
  $R_{N,\mathbf e}(a,b)$.  It allows us to derive from 
  estimate~\eqref{eq:20} that
\begin{equation}
  \label{eq:28}
  \Pp{\overline{\gamma} \lr[\overline{R_{N,\mathbf e}(a,b)}] \overline{Z}}>
  1-\eta.
 \end{equation}
 Now let $B_z$ and $B_{z'}$ be two neighbouring boxes. Let $\bar x$ be any
 vertex of $B_z$. By translation invariance, we get from~\eqref{eq:28}
 that
\begin{equation}
  \Pp{x+\overline{\gamma} \lr[\bar x+\overline{R_{N.\mathbf{e}}(a,b)}] \bar x
  +\overline{Z}}> 1-\eta.
\end{equation}
Here comes the key geometric observation: there exists $Z\in\mathcal
Z_{N,\mathbf e}(a,b,u,v)$ such that
\begin{equation}
  \bar x+\overline{Z}\subset B_{z'}+\overline{B(1)}.
\end{equation}
This is illustrated on Figures~\ref{fig:trickleft} and \ref{fig:trick} when $z=(0,0)$ and
$z'=(0,1)$.  Besides, $\bar x+\overline{R_N(a,b)}\subset C_{z'}$. Hence, by monotonicity,
we obtain that
\begin{equation}
   \Pp{\bar x+\overline{\gamma} \lr[C_{z'}]
    B_{z'}+\overline{B(1)}}>  1-\eta.
\end{equation}
\end{proof}

\begin{figure}[htbp]
  \centering
      \includegraphics[width=.75\textwidth]{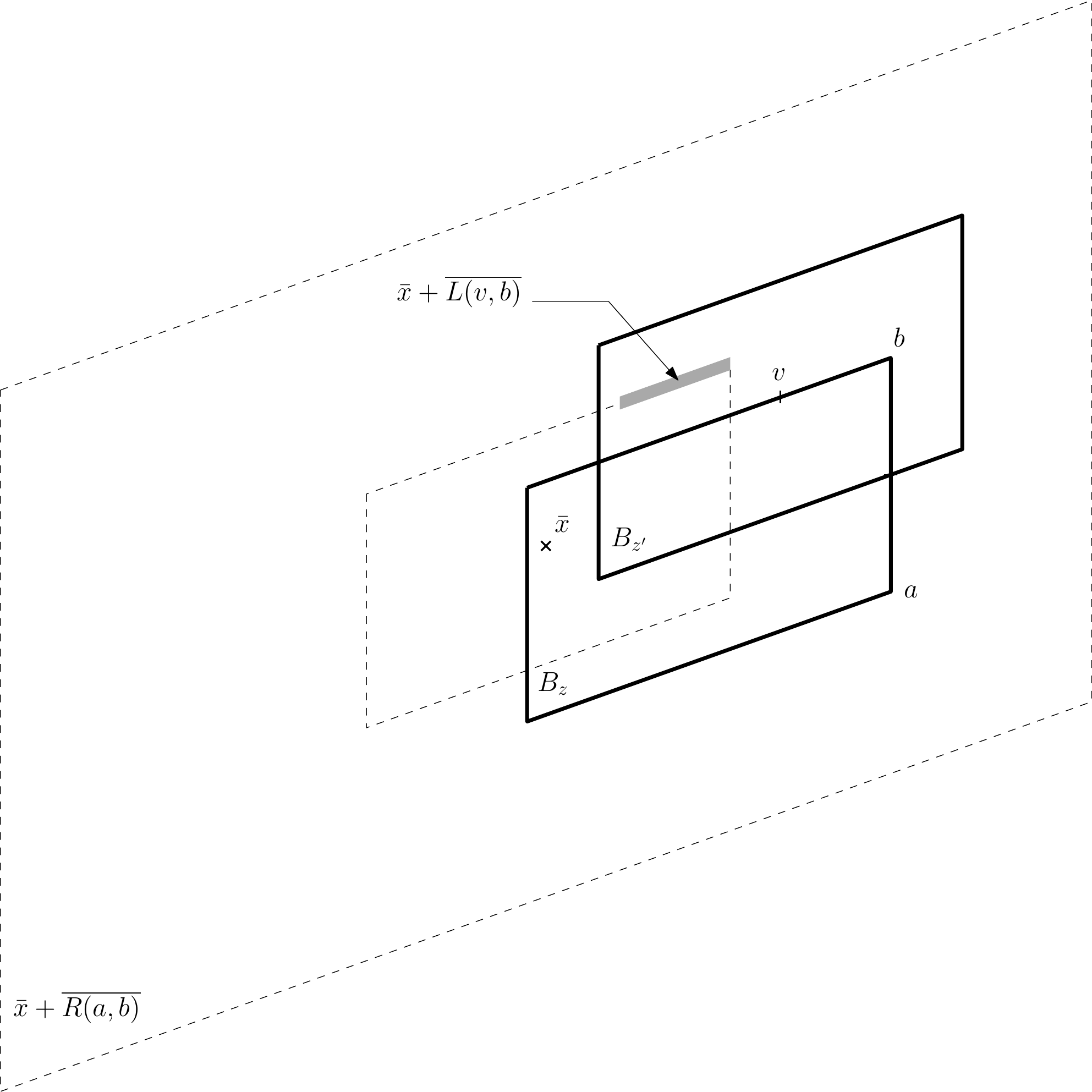}
   \caption{If $\bar x$ is on the left of the box
$B_z$, then $\bar x+\overline{L(v,b)}\subset B_{z'}$.}
  \label{fig:trickleft}
\end{figure}

\begin{figure}[htbp]
  \centering
      \includegraphics[width=.75\textwidth]{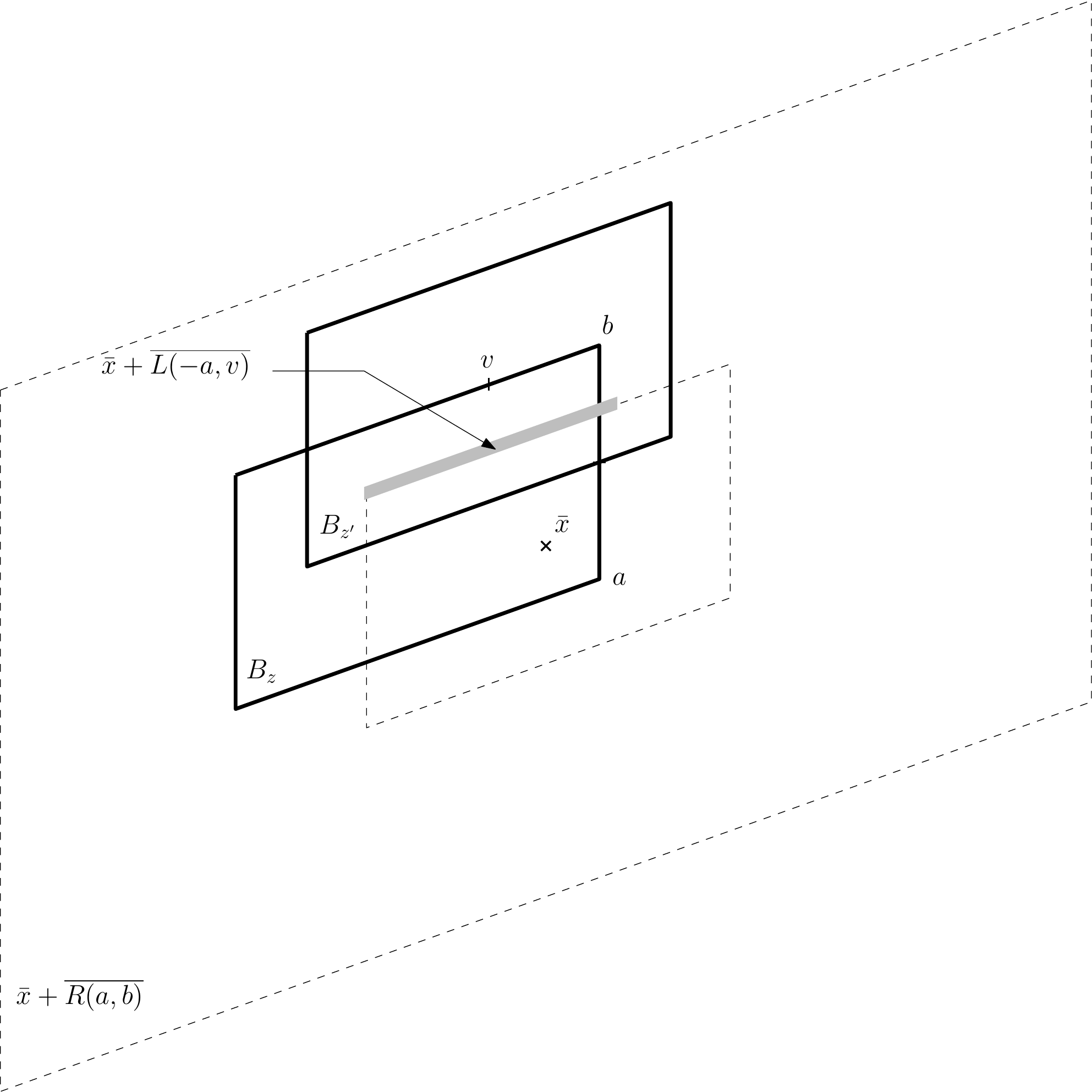}
   \caption{If $\bar x$ is on the right of the box
$B_z$, then $\bar x+\overline{L(-a,v)}\subset B_{z'}$.}
  \label{fig:trick}
\end{figure}

\subsection{Probabilistic setting}

Let $\omega_0$ be Bernoulli percolation of parameter $p$ on
$\overline{\mathcal G}$. In order to apply a ``sprinkling argument'',
we define for every $z\in \Z^2$ a sequence $(\xi^z(e))_{e \text{ edges
    in }C_z}$ of independent Bernoulli variables of parameter
$\tfrac\delta\kappa$. In other words, $\xi^z$ is a
$\tfrac\delta\kappa$-percolation on $C_z$. We assume that $\omega_0$
and all the $\xi^z$'s are independent. Lemma~\ref{sec:geom-sett-boxes}
implies that at most $\kappa+1$ Bernoulli variables are associated to
a given edge $e$: $\omega_0(e)$ and the $\xi^z(e)$'s for $z$ such that
$e\subset C_z$.

To state lemma~\ref{sec:lemma-1}, we also need the notion
of edge-boundary. The \emph{edge-boundary} of a set $A$ 
of vertices is the set of the edges of $\mathcal{G}$ with exactly one endpoint
in $A$. It is denoted by $\Delta A$.

\begin{lemma}\label{sec:lemma-1}
  Let $B_z$ and $B_{z'}$ be two neighbouring boxes. Let $H$ be a subset
  of $\overline{V}$. Let $(\omega(e))_{e\in E}$ be a family of
  independent Bernoulli variables of parameter $\Px{\omega(e)=1}\in
  [p,1)$ independent of $\xi^{z'}$. If there exists $\bar x\in B_z$ 
  and $\gamma\in\mathcal S(m)$ such that
  $\bar x+\overline\gamma \subset H$, then
 \begin{equation}
   \Px{ \left. H \lrc[C_{z'}]{\omega\vee
         \xi^{z'}} B_{z'} + \overline{B(1)} \: \right | \:
     \forall e\in \Delta H, \, \omega(e)=0 } \ge p_0.
  \end{equation}
\end{lemma}

\begin{proof}
  In all this proof, the marginals of $\omega$ are assumed to be
  Bernoulli random variables of parameter $p$. The more general
  statement of Lemma~\ref{sec:lemma-1} follows by a stochastic
  domination argument. The case $H\cap
  (B_{z'}+\overline{B(1)})\neq\emptyset$ being trivial, we assume that
  $H\cap(B_{z'}+\overline{B(1)})= \emptyset$.

  Let $W \subset \Delta H$
  be the (random) set of edges $\{\bar x,\bar y\}\subset C_{z'}$ such that
  \begin{enumerate}[(i)]
  \item $\bar x \in H$, $\bar y \in C_{z'} \setminus H$ and
  \item there is an $\omega$-open path joining $\bar y$ to
    $B_{z'}+\overline{B(1)}$, lying in $C_{z'}$, but using no edge
    with an endpoint in $H$.
  \end{enumerate}
  In a first step, we want to say that $|W|$ cannot be too small. The
  inclusions $\bar x+\overline\gamma\subset H \subset
  (B_{z'}+\overline{B(1)})^c$ imply that any $\omega$-open path from
  $\bar x+\overline\gamma$ to $B_{z'}+\overline{B(1)}$ must contain at
  least one edge of $W$. Thus, there is no $\omega$-open path
  connecting $\bar x+\overline\gamma$ to $B_{z'}+\overline{B(1)}$ in
  $C_{z'}$ when all the edges of $W$ are $\omega$-closed.
  Consequently, for any $t\in\N$, we have
  \begin{align}
    \Px{\left(\bar x+\overline\gamma\lrc[C_{z'}]{\omega} B_{z'}+
    \overline{B(1)}\right)^{\!\!c}\,} &\geq \Px{ \textrm{all
        edges in $W$ are $\omega$-closed} }\\
    &\geq {(1-p)}^t \Px{|W|\leq t}.
  \end{align}
  To get the last inequality above, remark that the random set $W$ is
  independent from the $\omega$-state of the edges in $\Delta
  H$. Using estimate~\eqref{eq:27}, it can be rewritten as
  \begin{equation}
    \label{eq:29}
     \Px{|W| \leq t} \leq \eta {(1-p)}^{-t}.
  \end{equation}
  We distinguish two cases. Either $W$ is small, which has a
  probability estimated by equation~\eqref{eq:29} above; or $W$ is
  large, and we use in that case that $B_{z'}+\overline{B(1)}$ is
  connected to $H$ as soon as one edge of $W$ is $\xi^{z'}$-open. The
  following computation makes it quantitative:
  \begin{align}
    &\Px{H \lrc[C_{z'}]{\omega\vee \xi^{z'}} B_{z'}+\overline{B(1)} \big |
      \forall e \in \Delta H, \: \omega(e)=0}\\
    &\quad \geq \Px{\big.\textrm{at least one edge of $W$ is $\xi^{z'}$-open}
      \big |
      \forall e \in \Delta H, \: \omega(e)=0} \\
    &\quad = \Px{\big.\textrm{at least one edge of $W$ is $\xi^{z'}$-open}}\\
    &\quad \geq \Px{\big.\textrm{at least one edge of $W$ is $\xi^{z'}$-open and
      } |W|>t} \\
    &\quad \geq 1 -\Px{\textrm{all the edges of $W$ are
        $\xi^{z'}$-open} \big | |W|>t} - \Px{\big.|W|\leq t}.
  \end{align}
  Using equation~\eqref{eq:29}, we conclude that, for any $t$,
  \begin{equation}
    \label{eq:30} \Px{H \lrc[C_{z'}]{\omega\vee \xi^{z'}}
A \big | \forall e \in \Delta H, \:
\xi^{z'}(e)=0}\geq 1-{(1-\delta/\kappa)}^t -\eta(1-p)^{-t}.
  \end{equation}
  Our choice of $\eta$ in~\eqref{eq:23} make the right hand side
  of~\eqref{eq:30} larger than $p_0$.
\end{proof}

\begin{lemma}
    \label{sec:probabilist-setting-1} With positive probability, the origin is connected
  to infinity in the configuration
    \[
     \omega_{\mathrm{total}}:=\omega_0\vee \bigvee_{z\in\Z^2} \xi^z.
    \]
\end{lemma}
Lemma~\ref{sec:probabilist-setting-1} concludes the proof of
Lemma~\ref{lem:renormalization} because $\omega_{\mathrm{total}}$ is
stochastically dominated by a $(p+\delta)$-percolation. Indeed,
$(\omega_{\mathrm{total}}(e))_e$ is an independent sequence of
Bernoulli variables such that, for any edge $e$,
\begin{equation}
  \Px{\omega_{\mathrm{total}}(e)=1}\geq 1-(1-p)(1- \delta/\kappa)
  ^{\kappa}\geq p+\delta.
\end{equation}
  \begin{proof}[Proof of Lemma~\ref{sec:probabilist-setting-1}]
    The strategy of the proof is similar to the one described in the
    original paper of Grimmett and Marstrand: we explore the Bernoulli
    variables one after the other in an order prescribed by the
    algorithm hereafter. During the exploration, we define
    simultaneously random variables on the graph
    $\overline{\mathcal{G}}$ and on the square lattice
    $\Z^2$.

\medbreak
\fbox{\begin{minipage}{0.9\textwidth}\begin{center}{\bf Algorithm} \end{center}

    \begin{itemize}
    \item[($0$)] Set $z(0)=(0,0)\in \Z^2$. Explore the connected
      component $H_0$ of the origin in $\mathcal G$ in the
      configuration $\omega_0$. Notice that only the edges of
      $H_0\cup\Delta H_0$ have been explored in order to determine
      $H_0$.
      \begin{itemize}
      \item If $H_0$ contains a path of $\mathcal S(m)$, set
        $X((0,0))=1$ and $(U_0,V_0)=(\{0\},\emptyset)$ and move  to ($t=1$).
      \item Else, set $X((0,0))=0$ and $(U_0,V_0)=(\emptyset,\{0\})$
        and move to ($t=1$).
      \end{itemize}

    \item[($t$)] Call \emph{unexplored} the vertices in $\Z^2 \setminus
      (U_t\cup V_t)$. Examine the set of unexplored vertices
      neighbouring an element of $U_t$. If this set is empty, define
      $(U_{t+1},V_{t+1})=(U_t,V_t)$ and move to ($t+1$). Otherwise,
      choose such an unexplored vertex $z_t$. In the configuration
      $\omega_{t+1}:=\omega_t \vee \xi^{z_t}$, explore the connected
      component $H_{t+1}$ of the origin.
      \begin{itemize}
      \item If $H_{t+1}\cap B_{z_t}\neq\emptyset$, which means in
        particular that $B_{z_t}$ is connected to $0$ by an
        $\omega_{t+1}$-open path, then set $X(z_t)=1$ and $ (U_{t+1},V_{t+1})=(U_t \cup
        \{z_t\}, V_t)$ and move to ($t+1$).
      \item  Else set $X(z_t)=0$ and $(U_{t+1},V_{t+1})=(U_t, V_t \cup
        \{z_t\})$ and move to ($t+1$).
      \end{itemize}

    \end{itemize}
  \end{minipage}}
\medbreak
\medskip  This algorithm defines in particular:
\begin{itemize}
\item a random process growing in the lattice $\Z^2$,
  $$S_0=(U_0,V_0),S_1=(U_1,V_1),\ldots$$
\item a random sequence $(X(z_t))_{t\geq 0}$.
\end{itemize}
Lemma~\ref{sec:lemma-1} ensures that for all $t\ge1$, whenever $z_t$ is
  defined,
\begin{equation}
  \label{eq:31}
\Px{X(z_t)=1\left|S_0,S_1,\ldots S_{t-1}\right.}
  \ge p_0 >\mathrm{p_c}^{\!\!\!\mathrm{site}}(\Z^2).
\end{equation}
Estimate~\eqref{eq:31} states that each time we explore a new site
$z_t$, whatever the past of the exploration is, we have a
sufficiently high probability of success: together with Lemma~1
of~\cite{gm}, it ensures that
\begin{equation}
   \Px{\left|U\right|=\infty}>0,
\end{equation}
where $U:=\bigcup_{t\geq 0}U_t$ is the set of $z_t$'s such that
$X(z_t)$ equals~$1$. For such $z_t$'s, we know that $B_{z_t}$
is connected to the origin of $\overline{\mathcal G}$ by an
$\omega_{t+1}$-open path. Hence, when $U$ is infinite,
there must exist an infinite open connected component in the configuration
\[\omega_0\vee \bigvee_{t\geq 0} \xi^{z_t},\] 
which is a subconfiguration of $\omega_{\mathrm{total}}$, and Lemma~\ref{sec:probabilist-setting-1} is established.
\end{proof}

\section*{Acknowledgements}
\label{sec:acknowledgements}

We are grateful to Vincent Beffara for valuable discussions,
helpful comments on the first versions of this paper, and more generally
for precious advice all along the project. We also thank Itai
Benjamini for useful dicussions and comments on the paper, Hugo
Duminil-Copin for having initiated this project and Micka\"el de la
Salle for pointing out the compactness argument of
lemma~\ref{lem:reduce}.

\subsection*{References}

\renewcommand\refname{ }

\vspace{2cm}
\raggedleft
  {\sc
  UMPA, Ens de Lyon
  \smallskip

  Lyon, France
   \smallskip}
 
     \texttt{sebastien.martineau@ens-lyon.fr}\\
 
    \texttt{vincent.tassion@ens-lyon.fr}

\end{document}